\documentclass[11pt,letterpaper]{amsart}

\usepackage{graphicx}
\usepackage{amsmath}
\usepackage{mathtools}
\usepackage{amssymb}
\usepackage{amsfonts}
\usepackage{amsthm}
\usepackage{cancel}
\usepackage[all]{xy}
\usepackage[neveradjust]{paralist}
\usepackage{enumitem}
\usepackage{multicol}
\usepackage{mathrsfs}
\usepackage{mathdots}
\usepackage[pagebackref,colorlinks=true,linkcolor=blue,citecolor=blue]{hyperref}
\usepackage{tikz-cd}
\usepackage{tikz}
\usepackage{blkarray}
\usepackage{scalerel,stackengine}

\UseRawInputEncoding

\usepackage{color}

\newcommand{\N}{\mathbb{N}}

\newcommand{\C}{\mathbb{C}}

\newcommand{\inv}{^{-1}}

\newcommand{\ol}{\overline}

\newcommand{\GL}{\mathrm{GL}}
\newcommand{\G}{\mathrm{G}}
\newcommand{\SO}{\mathrm{SO}}

\newcommand{\comment}[1]{}

\newtheorem{thm}{Theorem}[section]
\newtheorem{cor}[thm]{Corollary}
\newtheorem{lemma}[thm]{Lemma}
\newtheorem{prop}[thm]{Proposition}
\newtheorem {conj}[thm]{Conjecture}

\newtheorem {ques/conj}[thm]{Question/Conjecture}

\numberwithin{equation}{section}

\begin{document}

\title[Local Converse Theorem for split $\SO_{2l}$]{On the local converse Theorem for split $\SO_{2l}$}

\author{Alexander Hazeltine}
\address{Department of Mathematics\\
Purdue University\\
West Lafayette, IN, 47907, USA}
\email{ahazelti@purdue.edu}

\author{Baiying Liu}
\address{Department of Mathematics\\
Purdue University\\
West Lafayette, IN, 47907, USA}
\email{liu2053@purdue.edu}


\date{\today}

\dedicatory{}


\thanks{The research of the second named author is partially supported by the NSF Grants DMS-1702218, DMS-1848058, and by start-up funds from the Department of Mathematics at Purdue University.}

\begin{abstract}
    In this paper, we prove the local converse theorem for split even special orthogonal groups over a non-Archimedean local field of characteristic zero. This is the only case left on local converse theorems of split classical groups and the difficulty
is the existence of the outer automorphism. We apply new ideas of considering the summation of partial Bessel functions and overcome this
difficulty. As a direct application, we obtain a weak rigidity theorem for irreducible generic cuspidal representations of split even special orthogonal groups.  
\end{abstract}

\maketitle


\section{Introduction}

Let $\G$ be a connected reductive group.
Let $F$ be a non-Archimedean local field of characteristic zero and let $G=\G(F)$.
In the representation theory of $G$, converse theorems seek to uniquely identify a representation from its invariants. Among others, local gamma factors are important arithmetic invariants which play very important roles in the theory of Langlands program. 
More precisely, let $\pi$ be an irreducible generic representation of $G$. The family of local twisted gamma factors $\gamma(s, \pi \times \tau, \psi)$, for $\tau$ any irreducible generic representation of $\GL_n(F)$, $\psi$ an additive character of $F$ and $s\in\mathbb{C}$, can be defined using Rankin--Selberg convolution (e.g., \cite{JPSS83,Sou93,GRS98, K10, K12, K132, K13, ST15}) or the Langlands--Shahidi method (\cite{S84, Sha90}). The local converse problem is to determine which family of local twisted gamma factors will uniquely determine $\pi$. The following is the famous Jacquet's conjecture on the local converse problem.

\begin{conj}[Jacquet's conjecture on the local converse problem]\label{lcp}
Let $\pi,\pi'$ be irreducible generic representations
of $\GL_l(F)$. Suppose that they have the same central character.
If
\[
\gamma(s, \pi \times \tau, \psi) = \gamma(s, \pi' \times \tau, \psi),
\]
as functions of the complex variable $s$, for all irreducible
generic representations $\tau$ of $\GL_n(F)$ with $1 \leq n \leq 
[\frac{l}{2}]$, then $\pi \cong \pi'$.
\end{conj}

Conjecture \ref{lcp} has been proved by Chai (\cite{Cha19}), and by Jacquet and the second-named author (\cite{JL18}), independently, using different analytic methods. Hence, we have the local converse theorem for $\GL_l$. The local converse theorems for other classical groups also have been proved in recent years, mainly by Jiang-Soudry (\cite{JS03}, $\SO_{2l+1}$), by Zhang (\cite{Z18}, $\mathrm{Sp}_{2l}$, $\mathrm{U}_{2l}$, \cite{Z19}, $\mathrm{U}_{2l+1}$), and by Morimoto (\cite{Mor18}, $\mathrm{U}_{2l}$). For more references on local converse theorems, we refer to the introduction of \cite{LZ18}. 

The case left for the local converse theorems of split classical groups is $\SO_{2l}$. The difficulty is the existence of the outer automorphism. In this paper and in \cite{HL22}, we develop new ideas and overcome this difficulty, for split $\SO_{2l}$ over nonarchimedean local fields of characteristic zero and over finite fields of odd characteristic, respectively. 
More precisely, in this paper, we prove the following theorem. 

\begin{thm}[The Local Converse Theorem for split $\SO_{2l}$, generic case]\label{converse thm intro generic}
Let $\pi$ and $\pi^\prime$ be irreducible $\psi$-generic representations of split $\SO_{2l}(F)$ with the same central character $\omega$. If $$\gamma(s, \pi\times\tau,\psi)=\gamma(s, \pi^\prime\times\tau,\psi),$$ as functions of the complex variable $s$,
for all irreducible generic representations $\tau$ of $\GL_n(F)$ with $n\leq l,$
then $\pi\cong\pi'$ or $\pi\cong c\cdot\pi',$
 where $c$ is the outer automorphism. 
\end{thm}

We prove Theorem \ref{converse thm intro generic} via a reduction, applying similar ideas in \cite[\S 3.2]{JS03} (see Theorem \ref{thm reduction from generic to supercuspidal}), to the case of generic supercuspidal representations as follows. 

\begin{thm}[The Local Converse Theorem for split $\SO_{2l}$, supercuspidal case]\label{converse thm intro}
Let $\pi$ and $\pi^\prime$ be irreducible $\psi$-generic supercuspidal representations of split $\SO_{2l}(F)$ with the same central character $\omega$. If $$\gamma(s, \pi\times\tau,\psi)=\gamma(s, \pi^\prime\times\tau,\psi),$$ 
as functions of the complex variable $s$, for all irreducible generic representations $\tau$ of $\GL_n(F)$ with $n\leq l,$
then $\pi\cong\pi'$ or $\pi\cong c\cdot\pi',$
 where $c$ is the outer automorphism. 
\end{thm}

Theorems \ref{converse thm intro generic} and  \ref{converse thm intro} imply that local twisted gamma factors will not distinguish irreducible generic supercuspidal representations $\pi$ and $c \cdot \pi$ of $\SO_{2l}(F)$ (see also Corollaries \ref{conj gamma1}, \ref{conj gamma2} and \ref{conj gamma3}), which is a unique phenomenon for even special orthogonal groups among all the classical groups. This is consistent with the work of Arthur on the local Langlands correspondence and the local Langlands functoriality, and the work of Jiang and Soudry on local descent for $\SO_{2l}$ (see \cite{Art13} and \cite{JS12}). The analogue of Theorem \ref{converse thm intro} for quasi-split non-split $\SO_{2l}$ has more subtleties and has been proved by the first named author for both finite and local fields (\cite{Haz22a, Haz22b}). 

Now, we briefly introduce our new idea on proving Theorem \ref{converse thm intro}. 
As in many other proven cases of classical groups, we make use of the Howe vectors and the partial Bessel functions of $\pi$ which are particular Whittaker functions in the Whittaker model of $\pi$. 
More precisely, let $\mathcal{M}(\pi)$ be the set of matrix coefficients of $\pi$. 
Let $C_c^\infty(\SO_{2l},\omega)$ be the space of smooth compactly supported functions $f$ satisfying $f(zg)=\omega(z)f(g)$ for any $z\in Z$ and $g\in\SO_{2l}(F).$ 
Since $\pi$ is supercuspidal, 
$\mathcal{M}(\pi) \subset C_c^\infty(\SO_{2l},\omega)$.
$\psi$ induces a generic character on the upper triangular unipotent subgroup $U_{\SO_{2l}}$ (see \eqref{generic character}), still denoted by $\psi$.
Let $C^\infty(\SO_{2l},\omega,\psi)$ be the space of functions $W$ satisfying $W(zg)=\omega(z)W(g)$ and $W(ug)=\psi(u)W(g)$ for any $z\in Z,$ $u\in U_{\SO_{2l}}$, and $g\in\SO_{2l}(F),$ and furthermore that there exists an open compact subgroup $K$ of $\SO_{2l}(F)$ for which $W$ is invariant under right translation. Since $\pi$ is generic, we have a nonzero map
$$f \in C_c^\infty(\SO_{2l},\omega) \mapsto W^f \in C^\infty(\SO_{2l},\omega,\psi),$$ given by 
$$
W^f(g)=\int_{U_{\SO_{2l}}} \psi\inv(u)f(ug)du.
$$
Choose 
$f\in C_c^\infty(\SO_{2l},\omega)$ such that $W^f(I_{2l})=1.$
Then for any large positive integer $m$, one can associate a Howe vector to $W^f$, denoted by 
$B_m(g;f)$, $g \in \SO_{2l}(F)$, called a partial Bessel function (see \S \ref{Howe vectors} for more details).

Fix matrix coefficients $f\in\mathcal{M}(\pi)$ and $f'\in\mathcal{M}(\pi')$ such that $W^f(I_{2l})=W^{f'}(I_{2l})=1.$ The goal is to study the relation of $B_m(g,f)$ and $B_m(g,f')$ under the assumptions on the equality of local twisted gamma factors. First, we study the support of the partial Bessel functions and partition it based on Bruhat cells corresponding to $\mathrm{B}_n(\SO_{2l})$ for $n=1,\dots,l$ and $\mathrm{B}^c_l(\SO_{2l})$ (see Proposition \ref{Besselpart}), where 
$\mathrm{B}^c_l(\SO_{2l})$ is the $c$ conjugate of the set $\mathrm{B}_l(\SO_{2l})$. 
Then, we show that (in Corollary \ref{C(0)}) there exist functions $f_i \in C_c^\infty(\SO_{2l},\omega)$, $1 \leq i \leq l$, such that 
$$
B_m(g,f)-B_m(g,f')=\sum_{n=1}^{l} B_m(g,f_n),
$$
where $B_m(g,f_n)$ is supported on the Bruhat cells corresponding to $\mathrm{B}_n(\SO_{2l})$ for $n=1,\dots,l-2$, $B_m(g,f_{l-1})$ and $B_m(g,f_l)$ are supported on the Bruhat cells corresponding to
$\mathrm{B}_{l-1}(\SO_{2l}) \cup \mathrm{B}_{l}(\SO_{2l}) \cup \mathrm{B}^c_{l}(\SO_{2l})$, and the set of Bruhat cells corresponding to  $\mathrm{B}_{l-1}(\SO_{2l})$ is exactly the intersection of the supports of $B_m(g,f_{l-1})$ and $B_m(g,f_l)$. 

We show that (in Theorem \ref{l-2 theorem}), for $1 \leq k \leq l-2$, the local twisted gamma factors up to $k$ imply that 
$$
B_m(g,f)-B_m(g,f')=\sum_{n=k+1}^{l} B_m(g,f_i).
$$
Then, the local twisted gamma factors up to $l-1$ only imply that the sum of the partial Bessel functions 
$B_m(g,f_{l-1})+B_m(g,f_l)$ vanishes on part of the Bruhat cells corresponding to $\mathrm{B}_{l-1}(\SO_{2l})$ (see Theorem \ref{l-1 thm}). 
Comparing to the other cases considered in \cite{Z18, Z19}, instead of obtaining the vanishing of $B_m(g,f_{l-1})$ and $B_m(g,f_l)$ individually, or the vanishing of the summation, our new idea is to show that the local twisted gamma factors up to $l$ imply that 
$$
B_m(g,f_{l-1})+B_m(g,f_l)+B_m^c(g,f_{l-1})+B_m^c(g,f_l)=0,
$$
that is, 
$$B_m(g,f)-B_m(g,f')+B_m^c(g,f)-B_m^c(g,f')=0.$$
In other words, instead of considering the relation of $B_m(g,f)$ and $B_m(g,f')$, we consider the relation of $B_m(g,f)+B^c_m(g,f)$ and $B_m(g,f')+B^c_m(g,f')$. 
We record this as the following theorem. 

\begin{thm}[Theorem \ref{l thm}]\label{l thm intro}
Let $\pi$ and $\pi^\prime$ be irreducible supercuspidal $\psi$-generic representations of split $\SO_{2l}(F)$ with the same central character. Fix matrix coefficients $f\in\mathcal{M}(\pi)$ and $f'\in\mathcal{M}(\pi')$ such that $W^f(I_{2l})=W^{f'}(I_{2l})=1.$ If $$\gamma(s, \pi\times\tau,\psi)=\gamma(s, \pi^\prime\times\tau,\psi),$$ as functions of the complex variable $s$,
for all irreducible generic representations $\tau$ of $\GL_n$ with $1 \leq n\leq l,$
then we have that
$$B_m(g,f)-B_m(g,f')+B_m^c(g,f)-B_m^c(g,f')=0,$$
for any $g\in\SO_{2l}( F)$.
\end{thm}

By the uniqueness of Whittaker models, we obtain that $\pi\cong\pi'$ or $\pi\cong c\cdot\pi'$. Hence, Theorem \ref{converse thm intro} is proved. The proof of Theorem \ref{converse thm intro} is outlined in \S \ref{outline of the proof} and the full details are provided in \S \ref{proof of the local converse theorem}.

As a direct application of Theorem \ref{converse thm intro generic}, following the same argument as in \cite[Theorem 5.3]{JS03}, we obtain the following weak rigidity theorem for irreducible generic cuspidal automorphic representations of split $\SO_{2l}(\mathbb{A})$ (for the quasi-split analogue, see \cite{Haz22b}), where $\mathbb{A}$ is the ring of adeles for a number field. The proof is omitted.

\begin{thm}\label{rigidity}
Let $\Pi =\otimes_v \Pi_v$ and $\Pi'=\otimes_v \Pi'_v$ be two irreducible generic cuspidal automorphic representations of the split group $\SO_{2l}(\mathbb{A})$. If $\Pi_v \cong \Pi'_v$ for almost all places $v$, then $\Pi_v \cong \Pi'_v$ or $\Pi_v \cong c\cdot\Pi'_v$ for every place $v.$ 
\end{thm}

Theorem \ref{converse thm intro generic} and hence Theorem \ref{rigidity} have obtained independently by Haan-Kim-Kwon (\cite{HKK23}), using different method of theta correspondence, for both even and special even orthogonal groups. The advantage of our method is that it reflects the intrinsic properties of the partial Bessel functions of special even orthogonal groups. We remark that the $\mathrm{O}_{2l}$-analogue of Theorems \ref{converse thm intro generic} and \ref{rigidity} easily follow from those of $\SO_{2l}$. We also remark that Theorem \ref{converse thm intro generic} can also be obtained by applying the theory of local Langlands functoriality and the work of Arthur (\cite{Art13}).
Finally, we note that our method is independent of the work of Arthur (\cite{Art13}).

Following is the structure of this paper. In \S \ref{gps and reps}, we introduce the groups and representations considered in this paper. In \S \ref{sec: reduction}, we reduce the proof of Theorem \ref{converse thm intro generic} to that of Theorem \ref{converse thm intro}. 
In \S \ref{zeta integrals and gamma factor}, we recall the zeta integrals and gamma factors. In \S \ref{Partial Bessel functions}, we recall the theory of Howe vectors and partial Bessel functions and outline the proof of Theorem \ref{converse thm intro}. In
\S \ref{sections}, we construct a section which is used in calculating the zeta integrals. In \S \ref{proof of the local converse theorem}, we study the $\GL_n$ twists, $1 \leq n \leq l$, and show the relation between the $\GL_n$ twists and the support of the partial Bessel functions (Theorems \ref{l-2 theorem}, \ref{l-1 thm}, \ref{l thm}). Then, we prove Theorem \ref{l thm intro} and Theorem \ref{converse thm intro}. 

\subsection*{Acknowledgements}

The authors would like to thank Professor Dihua Jiang and Professor Freydoon Shahidi for their interests and constant support.
The authors also would like to thank Professor Qing Zhang and Professor Hiraku Atobe for helpful communications and comments.

\section{The groups and representations}\label{gps and reps}

Let $n, l \in \N$ and $F$ be a non-Archimedean local field of characteristic $0$. We fix a nontrivial unramified additive character $\psi$ of $F$ and let $\mathfrak{o}$  be the integers of $F$ with maximal ideal $\mathfrak{p}$ and uniformizer $\varpi.$

Let $\GL_n$ to be the group of matrices with entries in $F$ and non-zero determinant. Let $I_n$ be the identity element and we fix

$$J_n=\left(\begin{matrix}
0 & 0 & \cdots & 0 & 1 \\
0 & 0 & \cdots & 1 & 0 \\
\vdots & \vdots & \iddots & \vdots & \vdots \\
0 & 1 & \cdots & 0 & 0 \\
1 & 0 & \cdots & 0 & 0 \\
\end{matrix}\right).$$ We set $\SO_{n}=\{g\in\GL_{n} \, | \, \mathrm{det}(g)=1, {}^tgJ_{n} g = J_{n}\}$ to be the split special orthogonal groups.
Let $U_{G}$ be the subgroup of unipotent upper triangular matrices in $G=\GL_n,\SO_{n}$. Fix $B_{G}=T_{G}U_{G}$ to be the Borel subgroup of $G=\GL_n, \SO_n$ with split torus $T_G.$ 

Set $$
c=\mathrm{diag}(I_{l-1},
\left(\begin{matrix}
 0 & 1 \\
 1 & 0  
\end{matrix}\right), I_{l-1}).$$ Note that $c\notin \SO_{2l};$ however, $c \SO_{2l} c\inv=c\SO_{2l}c= \SO_{2l}$. Given a representation $\pi$ of $\SO_{2l}$ we define another representation $c\cdot\pi$ of $\SO_{2l}$ by $c\cdot\pi(g)=\pi(cgc).$

Following \cite{K13, K15} we fix certain embeddings of special orthogonal groups. These are needed to define the zeta integrals later. If $n<l$ we embed $\SO_{2n+1}$ into $\SO_{2l}$ via
$$
\left(
\begin{matrix}
A & B & C \\
D & E & K \\
L & P & Q 
\end{matrix} 
\right)
\mapsto
\mathrm{diag}(I_{l-n-1},
M^{-1}\left(\begin{matrix}
A & & B & C \\
 & 1 & & \\
D & & E & K \\
L & & P & Q 
\end{matrix}\right)
M, I_{l-n-1}),
$$
where $A$ and $Q$ are $n\times n$ matrices and 
$$
M=
\mathrm{diag}(I_n,\left(\begin{matrix}
    2 & -1 \\
    1 & \frac{1}{2}
    \end{matrix}\right), I_n).
    $$
This embeds $\SO_{2n+1}$ into the standard Levi subgroup of $\SO_{2l}$ which is isomorphic to $\GL_{l-n-1}\times\SO_{2n+2}$.

If $l=n$, we embed $\SO_{2l}$ into $\SO_{2l+1}$ via
$$
\left(\begin{matrix}
A & B \\
C & D
\end{matrix}\right)
\mapsto
M^{-1}
\left(\begin{matrix}
A & & B \\
 & 1 & \\
C & & D  
\end{matrix}\right)
M,
$$
where $A, B, C,$ and $D$ are $n\times n$ matrices and
$$
M=
\mathrm{diag}(I_{l-1},\left(\begin{matrix}
    \frac{1}{4} & \frac{1}{2} & \frac{-1}{2} \\
    \frac{1}{2} & 0 & 1 \\
    \frac{-1}{2} & 1 & 1
    \end{matrix}\right), I_{l-1}).
$$
Note that this embedding takes the torus $T_{\SO_{2l}}$ to a torus in $\SO_{2l+1}$, but not the standard torus. Indeed, the embedding sends  $t=\mathrm{diag}(t_1,\dots, t_l, t_l\inv,\dots,t_1\inv)$ to
$$
\mathrm{diag}(s,\left(\begin{matrix}
\frac{1}{2}+\frac{1}{4}(t_l+t_l\inv) & \frac{1}{2}(t_l-t_l\inv) & 2(\frac{1}{2}-\frac{1}{4}(t_l+t_l\inv)) \\
\frac{1}{4}(t_l-t_l\inv) & \frac{1}{2}(t_l+t_l\inv) & \frac{-1}{2}(t_l-t_l\inv) \\
\frac{1}{2}(\frac{1}{2}-\frac{1}{4}(t_l+t_l\inv)) & \frac{-1}{4}(t_l-t_l\inv) & \frac{1}{2}+\frac{1}{4}(t_l+t_l\inv)
\end{matrix}\right), s^*),
$$
where $s=\mathrm{diag}(t_1, t_2, \dots, t_{l-1})$.

We define what it means for a representation to be generic. Recall that $U_{\GL_n}$ and $U_{\SO_{2l}}$ are the subgroups of unipotent upper triangular matrices in $\GL_n$ and $\SO_{2l}$ respectively and that we fixed an additive nontrivial character $\psi$ of $F$. We abuse notation and define a generic character, which we  also call $\psi$, on $U_{\GL_n}$ and $U_{\SO_{2l}}$. For $u=(u_{i,j})_{i,j=1}^n\in U_{\GL_n}$, we set $\psi(u)=\psi\left(\sum_{i=1}^{l-1} u_{i,i+1}\right).$ For $u=(u_{i,j})_{i,j=1}^l\in U_{\SO_{2l}}$ we set 
\begin{equation}\label{generic character}
    \psi(u)= 
   \psi\left(\sum_{i=1}^{l-2} u_{i,i+1}+\frac{1}{4}u_{l-1,l}-\frac{1}{2} u_{l-1,l+1}  \right).
\end{equation}
Let $\psi_n=1$ for $n\leq l-2,$ $\psi_{l-1}=\frac{1}{4}$, and $\psi_l=\frac{-1}{2}$ denote the coefficients of this character. 

We say an irreducible representation $\pi$ of $\SO_{2l}$ is $\psi$-generic if $$\mathrm{Hom}_{U_{\SO_{2l}}}(\pi,\psi)\neq 0.$$
Similarly, we say an irreducible representation $\tau$ of $\GL_{n}$ is $\psi$-generic if $$\mathrm{Hom}_{U_{\GL_{n}}}(\tau,\psi)\neq 0.$$
A nonzero intertwining operator in these spaces is called a Whittaker functional and it is well known that Whittaker functionals are unique up to scalars (by the uniqueness of Whittaker models). Fix $\Gamma\in\mathrm{Hom}_{U_{\SO_{2l}}}(\pi,\psi)$ to be a nonzero Whittaker functional. For $v\in\pi$, let $W_v(g)=\Gamma(\pi(g)v)$ for any $g\in\SO_{2l}$ and set $\mathcal{W}(\pi,\psi)=\{W_v \, | \, v\in\pi\}.$ $\mathcal{W}(\pi,\psi)$ is called the $\psi$-Whittaker model of $\pi.$ By Frobenius reciprocity, $\mathrm{Hom}_{U_{\SO_{2l}}}(\pi,\psi)\cong\mathrm{Hom}_{SO_{2l}}(\pi,\mathrm{Ind}_{U_{\SO_{2l}}}^{\SO_{2l}}(\psi)).$ Thus, $\pi$ can be realized as a subrepresentation of $\mathrm{Ind}_{U_{\SO_{2l}}}^{\SO_{2l}}(\psi)$ via the map $\pi \rightarrow \mathcal{W}(\pi,\psi)$ given by $v\mapsto W_v.$ Moreover, by the uniqueness of Whittaker models, this subrepresentation occurs with multiplicity one inside $\mathrm{Ind}_{U_{\SO_{2l}}}^{\SO_{2l}}(\psi)$. We also note that the analogous results hold for $\psi$-generic representations $\tau$ of $\GL_n.$

Let  $Q_n=L_n V_n$ be the standard Siegel parabolic subgroup of $\SO_{2n+1}$ with Levi subgroup $L_n\cong \GL_n.$ For $a\in\GL_n$ we let $l_n(a)=\mathrm{diag}(a,1,a^*)\in L_n$ where $a^*=J_n{}^ta^{-1}J_n.$ Let $\tau$ be an irreducible $\psi^{-1}$-generic representation of $\GL_{n}$, $s\in\C$, and set $I(\tau, s)=\mathrm{Ind}_{Q_n}^{\SO_{2n+1}}(\tau |\mathrm{det}|^{s-\frac{1}{2}}).$ An element $\xi_s\in I(\tau,s)$ is a function $\xi_s:\SO_{2n+1}\rightarrow\tau$ satisfying $$
\xi_s(l_n(a)ug)=|\mathrm{det}(a)|^{s-\frac{1}{2}}\tau(a)(\xi_s(g)), \forall a\in\GL_n, u\in V_n, g\in\SO_{2n+1},
$$ and is right translation invariant by some compact open subgroup.
Fix a nonzero homomorphism $\Lambda_\tau \in \mathrm{Hom}_{U_{\GL_{n}}}(\tau,\psi^{-1})$. For $\xi_s\in I (\tau,s)$, we let $f_{\xi_s} : \SO_{2n+1}\times\GL_n \rightarrow \C$ be the function given by $$
f_{\xi_s}(g,a)=\Lambda_\tau(\tau(a)\xi_s(g)).
$$
Let $I(\tau,s,\psi^{-1})$ be the space of functions generated by $f_{\xi_s}, \xi_s\in I(\tau,s).$ Note that for $f_s\in I(\tau,s,\psi^{-1}),$ we have 
$$
f_s(g,ua)=\psi^{-1}(u)f_s(g,a), \forall g\in\SO_{2n+1}, u\in U_{\GL_n}, a\in\GL_n.
$$
We  also let $\tau^*$ be the contragradient representation of $\GL_n$ defined by $\tau^*(a)=\tau(a^*).$

\section{Zeta integrals and the gamma factor}\label{zeta integrals and gamma factor}

In this section, we recall the definition of the zeta integrals in \cite{K13}. In the case $n=l$, we already have the notation needed; however, we need to define another unipotent subgroup and character when $n<l.$ 

Suppose that $n<l.$ Let $P_{l-n-1}=M_{l-n-1}N_{l-n-1}$ be the standard parabolic subgroup of $\SO_{2l}$ with Levi subgroup, $M_{l-n-1}$, isomorphic to $\GL_{l-n-1}\times\SO_{2n+2}.$ We embed $U_{\GL_{l-n-1}}$ inside $\GL_{l-n-1}$ which is realized in $M_{l-n-1}.$ Then we define the unipotent subgroup $N^{l-n}=U_{\GL_{l-n-1}}N_{l-n-1}.$ That is,

$$
N^{l-n}=\left\{\left(\begin{matrix}
u_1 & v_1 & v_2 \\
 & I_{\SO_{2n+2}} & v_1^\prime \\
 & & u_1^*
\end{matrix}\right) \in \SO_{2l} \, | \, u_1\in U_{\GL_{l-n-1}} \right\}.
$$
For $v=(v_{i,j})\in N^{l-n}$ we define a character, which we also call $\psi,$ of $N^{l-n}$ by

$$
\psi(v)= 
   \psi\left(\sum_{i=1}^{l-n-1} v_{i,i+1}+\frac{1}{4}v_{l-n-1,l}-\frac{1}{2} v_{l-n-1,l+1}   \right) .
$$
Note that this character is trivial when $n=l-1.$ Let $H=\SO_{2n+1}N^{l-n}$ where $\SO_{2n+1}$ is realized via the embedding into $\SO_{2n+2}$ inside $M_{l-n-1}$ and extend $\psi$ trivially across $\SO_{2n+1}$ so that $\psi$ is a character of $H$.

Let $\pi$ be an irreducible $\psi$-generic supercuspidal representation of $\SO_{2l}$, $\tau$ be a $\psi\inv$-generic representation of $\GL_n,$  $W\in\mathcal{W}(\pi,\psi),$ and $f_s\in I(\tau,s,\psi^{-1})$. We define the zeta integrals, $\Psi(W,f_s).$ Note that \cite{K13} defines integrals for any $n$ and $l$; however, we  only need the case of $n\leq l$ for the converse theorem and so we do not consider the case of $n>l$.

First, suppose that $n=l.$ Then we define
$$
\Psi(W,f_s)=\int_{ U_{\SO_{2l}}\backslash\SO_{2l}} W(g)f_s(w_{l,l}g, I_{l})dg,
$$
where 
$$
w_{l,l}=\left(\begin{matrix}
 \frac{1}{2}\cdot I_{l} & &  \\
  & 1 & \\
  & & 2\cdot I_l
\end{matrix}
\right)\in\SO_{2l+1}.$$
For any $g\in\SO_{2l},$ the integral satisfies \begin{equation}\label{zetan=l}
    \Psi(g\cdot W,g\cdot f_s)=\Psi(W,f_s).
\end{equation}

Next, suppose that $n<l.$ Then we define

$$
\Psi(W,f_s)=\int_{ U_{\SO_{2n+1}}\backslash\SO_{2n+1}} \left(\int_{r\in R^{l,n}} W(r w^{l,n} g (w^{l,n})\inv)dr\right)   f_s(g, I_{l})dg,
$$
where
$$
w^{l,n}=\left(\begin{matrix}
 & I_n & & & \\
 I_{l-n-1} & & & & \\
 & & I_2 & & \\
 & & & & I_{l-n-1} \\
 & & & I_n &
\end{matrix}
\right)\in\SO_{2l},$$
and
$$
R^{l,n}=\left\{\left(\begin{matrix}
I_n & & & & \\
x & I_{l-n-1} & & & \\
& & I_2 & & \\
& & & I_{l-n-1} & \\
& & & x^\prime & I_n
\end{matrix}\right)\in\SO_{2l}\right\}.
$$
The integral satisfies the property \begin{equation}\label{zetan<l}
    \Psi((gn)\cdot W,g\cdot f_s)=\psi\inv(n)\Psi(W,f_s),
\end{equation} for any $g\in\SO_{2n+1}$ and $n\in N^{l-n}.$ Note that in the case $n<l$, our integral differs from \cite{K13} slightly. The difference is a right translation of the Whittaker functional by $(w^{l,n})\inv.$

We introduce the standard intertwining operator $M(\tau,s,\psi^{-1})$. Note that we do not use the normalized version of the intertwining operator. We define $M(\tau,s,\psi^{-1}): I(\tau,s,\psi^{-1})\rightarrow I(\tau^*,1-s,\psi^{-1})$ via
$$
M(\tau,s,\psi^{-1})f_s(h,a)=\int_{V_n} f_s(w_n u h, d_n a^*)du,
$$
where $w_n=\left(\begin{matrix}
 & & I_n \\
 & (-1)^n & \\
 I_n & &
\end{matrix}\right)$, $d_n=\mathrm{diag}(-1,1,-1,\dots,(-1)^n)\in\GL_n$, and $V_n$ is the unipotent radical of the parabolic subgroup $Q_n=L_n V_n$ in $\SO_{2n+1}$ where $L_n\cong \GL_n$, and $a^*=J_n {}^t a^{-1} J_n.$

Outside of a discrete subset of $s$, the set of bilinear forms satisfying equations (\ref{zetan=l}) or (\ref{zetan<l}) is at most $1$ dimensional. This follows from the results of \cite{AGRS} and \cite{MW12} (see \cite{GGP12a}). When $n=l$, this was already known in \cite{GPSR}.

Therefore, we can define a proportionality factor, $\gamma(s, \pi\times\tau, \psi)$, called the $\gamma$-factor, such that
$$
\gamma(s,\pi\times\tau, \psi)\Psi(W,f_s)=\Psi(W,M(\tau,s,\psi^{-1})f_s).
$$
We refer to these integrals and $\gamma$-factors as the twists by $\GL_n.$

Note that our $\gamma$-factor differs from that of \cite{K13} slightly. Specifically, it differs by a factor depending only on $\tau$ and $\psi$. If we let $\gamma'(s,\pi\times\tau, \psi)$ be the gamma factor of \cite{K13}, then 
$$
\gamma(\tau, Sym^2,\psi,2s-1)\gamma(s,\pi\times\tau, \psi)=c(n,\tau,\psi,\gamma)\gamma'(s,\pi\times\tau, \psi),
$$
where for $n<l$ we have $c(n,\tau,\psi,\gamma)=\omega_\tau(\gamma)^2 |\gamma|^{2n(s-\frac{1}{2})}$ and $c(l,\tau,\psi,\gamma)=1.$ Here 
$\gamma(\tau, Sym^2,\psi,2s-1)$ is Shahidi's local coefficient, $ Sym^2$ is the symmetric square representation, and $\omega_\tau$ is the central character of $\tau.$ Thus, if the local converse theorem holds for $\gamma(s,\pi\times\tau, \psi)$, then it also  holds for $\gamma'(s,\pi\times\tau, \psi).$

\section{Reduction to the generic supercuspidal case}\label{sec: reduction}

In this section, we show that the local converse theorem for irreducible generic representations of $\SO_{2l}(F)$ follows from that of irreducible generic supercuspidal representations of $\SO_{2l}(F)$ (see Theorem \ref{thm reduction from generic to supercuspidal} below), applying similar ideas in \cite[\S 3.2]{JS03}. The argument relies on knowing the real poles and zeroes of certain local twisted gamma factors. For $\GL_n$, this is known due to \cite[Propositon 3.1]{JS03}. For $\SO_{2l}$, we establish a similar result (see Lemma \ref{lem poles of gamma SO_2l} below) using the functorial lift of \cite{CKPSS04}.

\begin{thm}[{\cite[Proposition 7.2 and Theorem 7.3]{CKPSS04}}]\label{thm functorial lift}
Suppose that $\pi^0$ is an irreducible unitary supercuspidal representation of $\SO_{2l}(F).$ Then, $\pi^0$ has a functorial lift to $\GL_{2l}(F)$ which is a generic representation of $\GL_{2l}(F)$ of the form
$$
\Pi^0=\mathrm{Ind}_{P}^{\GL_{2n}(F)} (\Pi^0_1\otimes\cdots\otimes\Pi_d^0)
$$
where for each $j=1,\dots,d,$ $\Pi^0_j$ is an irreducible self-dual representation of $\GL_{l_j}(F)$ where $\sum_{j=1}^d l_j=2l$ and $P$ is the appropriate standard parabolic subgroup. Moreover, $\Pi_i^0\not\cong\Pi_j^0$ for any $i\neq j.$ In particular, for any irreducible generic representation $\tau$ of $\GL_n(F)$ with $n\leq l,$ \begin{equation}\label{eqn gamma lift}
    \gamma(s, \pi^0\times\tau,\psi)=\gamma(s, \Pi^0\times\tau,\psi).
\end{equation}
\end{thm}

The following lemma describes all the possible real poles and zeroes of the $\gamma$-factor for an irreducible unitary supercuspidal representation of $\SO_{2l}(F).$

\begin{lemma}\label{lem poles of gamma SO_2l}
Suppose that $\pi^0$ is an irreducible unitary supercuspidal representation of $\SO_{2l}(F).$ Then, for any irreducible generic representation $\tau$ of $\GL_n(F)$ with $n\leq l,$ $\gamma(s, \pi^0\times\tau,\psi)$ only has a possible real pole (resp. real zero) at $s=1$ (resp. $s=0$). Moreover, the real pole and zero are simple and only occurs when $\tau$ is a supercuspidal representation occurring in the supercuspidal support of the functorial lift of $\pi^0$ (see Theorem \ref{thm functorial lift}). In particular, if there is a pole or zero, then $\tau$ is self-dual.
\end{lemma}

\begin{proof} Recall from Theorem \ref{thm functorial lift} that $\pi^0$ has a functorial lift to $\GL_{2n}(F).$ We continue with the notation of that theorem.
By the multiplicativity of $\gamma$-factors of general linear groups (\cite[Theorem 3.1]{JPSS83}),
\begin{equation}\label{eqn GL_n multiplicativity}
    \gamma(s, \Pi^0\times\tau,\psi)=\prod_{j=1}^d\gamma(s, \Pi^0_j\times\tau,\psi).
\end{equation}
From \cite[Proposition 3.1]{JS03}, for each $j=1,\dots,d,$ the only possible real pole (resp. real zero) of $\gamma(s, \Pi^0_j\times\tau,\psi)$ is at $s=1$ (resp. $s=0$) and $\tau=\Pi^0_j$ (note that $\Pi^0_j$ is self-dual). Thus we see from \eqref{eqn gamma lift} and \eqref{eqn GL_n multiplicativity} that $\gamma(s, \pi^0\times\tau,\psi)$ only has a possible real pole (resp. real zero) at $s=1$ (resp. $s=0$) and if and $\tau=\Pi^0_j$. Moreover, $\Pi^0_i\not\cong \Pi^0_j$ for any $i\neq j$ and so the poles and zeroes are simple. This proves the lemma.
\end{proof}

By Jacquet's subquotient theorem (\cite{Jac75}), there exists a standard parabolic subgroup $Q$ of $\SO_{2l}$ with Levi part isomorphic to
$$
\GL_{l_1}\times\cdots\times\GL_{l_m}\times\SO_{2l_0},
$$
where $l=\sum_{i=0}^m l_i$
such that $\pi$ is a subquotient of the induced representation 
$$
\tau^1\nu^{z^1}\times\cdots\times\tau^m\nu^{z^m}\rtimes\pi^0,
$$
where, for $i=1,\dots,m$ $\tau^i$ is an irreducible unitary supercuspidal representation of $\GL_{l_i}(F)$, $z^1\geq\dots\geq z^m\geq 0$ (we assume here that the $z^i$'s are all real without loss of generality) and $\pi^0$ is an irreducible generic supercuspidal representation of $\SO_{2l_0}(F).$ Note that the superscripts are indices here, not exponents. This data is called the supercuspidal support of $\pi.$ One could also arrive at a similar result using the classifications of \cite{JL14}. We show that the the local converse theorem for irreducible generic representations of $\SO_{2l}(F)$ follows from the local converse theorem for irreducible generic supercuspidal representations of $\SO_{2l}(F)$ by using the supercuspidal support along with the multiplicativity of $\gamma$-factors.

\begin{thm}\label{thm reduction from generic to supercuspidal}
Assume that Theorem \ref{converse thm intro} holds. Let $\pi_1$ and $\pi_2$ be irreducible $\psi$-generic representations of split $\SO_{2l}(F)$ with the same central character $\omega$. If $$\gamma(s, \pi_1\times\tau,\psi)=\gamma(s, \pi_2\times\tau,\psi),$$ 
for all irreducible generic representations $\tau$ of $\GL_n(F)$ with $n\leq l,$
then $\pi_1\cong\pi_2$ or $\pi_1\cong c\cdot\pi_2.$
\end{thm}

\begin{proof}
Suppose that for $j=1,2,$ $\pi_j$ is a subquotient of
$$
\tau^1_j\nu^{z^1_j}\times\cdots\times\tau^{r_j}_j\nu^{z^{r_j}_j}\rtimes\pi_j^0,
$$
where $\tau^1_j,\dots,\tau^{r_j}_j$ are irreducible unitary supercuspidal representations of general linear groups, $z^1_j\geq\cdots\geq z^{r_j}_j\geq 0$, and $\pi_j^{0}$ is a generic irreducible supercuspidal representation of an even special orthogonal group. Note that the superscripts here are indices, not exponents.
By multiplicativity of the $\gamma$-factor (\cite[Theorem 1]{K13}), we have for $j=1,2,$
\begin{equation}\label{eqn gamma decomp}
    \gamma(s,\pi_j\times\tau, \psi)=\gamma(s,\pi_j^{0}\times\tau, \psi)\prod_{i=1}^{r_j} \gamma(s+z^i_j,\tau^i_j\times\tau, \psi)\prod_{i=1}^{r_j} \gamma(s-z^i_j,(\tau^i_j)^*\times\tau, \psi).
\end{equation}
We begin by showing that $\gamma(s,\pi_1^{0}\times\tau, \psi)=\gamma(s,\pi_2^{0}\times\tau, \psi)$.

First, suppose that $r_1=0.$  That is, $\gamma(s,\pi_1\times\tau, \psi)=\gamma(s,\pi_1^{0}\times\tau, \psi).$ Assume that $r_2>0$ for contradiction. By Lemma \ref{lem poles of gamma SO_2l}, for $j=1,2,$ $\gamma(s,\pi_j^{0}\times\tau, \psi)$ has no zeroes for $\mathrm{Re}(s)>0.$ By \cite[Proposition 3.1]{JS03},
$$
\prod_{i=1}^{r_2} \gamma(s+z^i_2,\tau^i_2\times\tau, \psi)\prod_{i=1}^{r_2} \gamma(s-z^i_2,(\tau^i_2)^*\times\tau, \psi)
$$
has a pole at $1+z^1_2$ when $\tau=\tau_2^1$ and hence so does  $\gamma(s,\pi_2\times\tau, \psi).$ 
Thus, $\gamma(s,\pi_1^{0}\times\tau, \psi)$ has a pole at $1+z_2^1$ when $\tau=\tau_2^1$. By Lemma \ref{lem poles of gamma SO_2l}, the only possible real pole of $\gamma(s,\pi_1^{0}\times\tau, \psi)$ is at $s=1.$ Thus, we must have $z_2^1=z_2^2=\cdots=z_2^{r_2}=0$ and
$$
\gamma(s,\pi_1^{0}\times\tau, \psi)=\gamma(s,\pi_2^{0}\times\tau, \psi)\prod_{i=1}^{r_2} \gamma(s,\tau^i_2\times\tau, \psi)\prod_{i=1}^{r_2} \gamma(s,(\tau^i_2)^*\times\tau, \psi).
$$
By Lemma \ref{lem poles of gamma SO_2l}, if $\gamma(s,\pi_1^{0}\times\tau, \psi)$ has a pole, then it is simple and $\tau$ must be self-dual. Since we know
$\gamma(s,\pi_1^{0}\times\tau, \psi)$ has a pole when $\tau=\tau_2^1,$ we have $\tau_2^1$ is also self-dual. Thus, the pole at $s=1$ of $\gamma(s,\pi_2\times\tau_2^1, \psi)$ is of order at least 2, while the pole at $s=1$ of $\gamma(s,\pi_1\times\tau_2^1, \psi)$ is simple. However, by assumption $\gamma(s,\pi_1\times\tau_2^1, \psi)=\gamma(s,\pi_2\times\tau_2^1, \psi)$ and hence we have a contradiction. Therefore $r_1=r_2=0$ and $\gamma(s,\pi_1^{0}\times\tau, \psi)=\gamma(s,\pi_2^{0}\times\tau, \psi)$ as claimed.

Suppose now that $r_1>0.$ Without loss of generality, we also assume that $r_2\geq r_1.$ We have
\begin{align}\label{eqn r_1>0 gammas}
    &\gamma(s,\pi_1^{0}\times\tau, \psi)\prod_{i=1}^{r_1} \gamma(s+z^i_1,\tau^i_1\times\tau, \psi)\prod_{i=1}^{r_1} \gamma(s-z^i_1,(\tau^i_1)^*\times\tau, \psi) \\
    =&\gamma(s,\pi_2^{0}\times\tau, \psi)\prod_{i=1}^{r_2} \gamma(s+z^i_2,\tau^i_2\times\tau, \psi)\prod_{i=1}^{r_2} \gamma(s-z^i_2,(\tau^i_2)^*\times\tau, \psi). \nonumber
\end{align}
We will prove that $\gamma(s,\pi_1^{0}\times\tau, \psi)=\gamma(s,\pi_2^{0}\times\tau, \psi).$ 

 By Lemma \ref{lem poles of gamma SO_2l}, for $j=1,2,$ $\gamma(s,\pi_j^{0}\times\tau, \psi)$ has no zeroes for $\mathrm{Re}(s)>0$ and only a possible pole at $s=1.$ By \cite[Proposition 3.1]{JS03}, for $j=1,2,$
$$
\prod_{i=1}^{r_j} \gamma(s+z^i_j,\tau^i_j\times\tau, \psi)\prod_{i=1}^{r_j} \gamma(s-z^i_j,(\tau^i_j)^*\times\tau, \psi)
$$
has a pole at $1+z^1_j$ when $\tau=\tau_2^j.$ Thus, if $z^1_j>0,$ we must have $z^1_1=z^1_2$ and $\tau_2^1=\tau_2^2.$ Note that if $z^1_j>0$ for some $j$, then we must have $z^1_j>0$ for both $j=1,2$ to ensure that both $\gamma$-factors have the pole at $1+z^1_j.$ 

Continuing in this manner, we find that if $z_j^i>0$ for some $j$ and $i$, then $z_1^i=z_2^i$ and $\tau_1^i=\tau_2^i.$ Hence we may cancel out any $\gamma$-factors occurring in \eqref{eqn r_1>0 gammas} with $z_j^i>0.$ Hence we may assume that every $z_j^i=0.$ We have
\begin{align}\label{eqn r_1>0 gammas and z_j^i=0}
    &\gamma(s,\pi_1^{0}\times\tau, \psi)\prod_{i=1}^{r_1} \gamma(s,\tau^i_1\times\tau, \psi)\prod_{i=1}^{r_1} \gamma(s,(\tau^i_1)^*\times\tau, \psi) \\
    =&\gamma(s,\pi_2^{0}\times\tau, \psi)\prod_{i=1}^{r_2} \gamma(s,\tau^i_2\times\tau, \psi)\prod_{i=1}^{r_2} \gamma(s,(\tau^i_2)^*\times\tau, \psi). \nonumber
\end{align}
By Lemma \ref{lem poles of gamma SO_2l}, if $\gamma(s,\pi_j^{0}\times\tau, \psi)$ has a pole, then $\tau$ must be self-dual. Suppose that for some $i\in\{1,\dots,r_1\}$ $\tau_1^i$ is not self-dual. By \cite[Proposition 3.1]{JS03}, $\gamma(s,\tau^i_1\times\tau^i_1, \psi)$ contributes a pole at $s=1$ to $\gamma(s,\pi_1\times\tau^i_1, \psi).$ From \eqref{eqn r_1>0 gammas and z_j^i=0}, it follows that we must have $\tau_1^i=\tau_2^k$ or $\tau_1^i=(\tau_2^k)^*$ for some $k\in\{1,\dots,r_2\}.$ In either case, we may continue in this manner and cancel out any $\gamma$-factors occurring in \eqref{eqn r_1>0 gammas and z_j^i=0} for which $\tau_j^i$ is not self-dual.

Therefore, we may assume that every $\tau_j^i$ occurring in \eqref{eqn r_1>0 gammas and z_j^i=0} is self-dual. We have
\begin{align}\label{eqn r_1>0 gammas and self-dual}
    \gamma(s,\pi_1^{0}\times\tau, \psi)\prod_{i=1}^{r_1} \gamma(s,\tau^i_1\times\tau, \psi)^2 
    =\gamma(s,\pi_2^{0}\times\tau, \psi)\prod_{i=1}^{r_2} \gamma(s,\tau^i_2\times\tau, \psi)^2.
\end{align}
By \cite[Proposition 3.1]{JS03}, $\gamma(s,\tau^1_1\times\tau^1_1, \psi)^2$ contributes a pole of order 2 at $s=1$ to \eqref{eqn r_1>0 gammas and self-dual}. Since every pole occurring in $\gamma(s,\pi_2^{0}\times\tau, \psi)$ is simple by Lemma \ref{lem poles of gamma SO_2l}, we must have $\tau_1^1=\tau_2^k$ for some $k\in\{1,\dots,r_2\}.$ Thus we may cancel out $\gamma(s,\tau^1_1\times\tau^1_1, \psi)^2$ with $\gamma(s,\tau^k_2\times\tau^k_2, \psi)^2$ in \eqref{eqn r_1>0 gammas and self-dual}. Since $r_2\geq r_1,$ by continuing in this manner we may assume that $r_1=0$ in \eqref{eqn r_1>0 gammas and self-dual}. That is, \begin{align*}
    \gamma(s,\pi_1^{0}\times\tau, \psi)
    =\gamma(s,\pi_2^{0}\times\tau, \psi)\prod_{i=1}^{r_2} \gamma(s,\tau^i_2\times\tau, \psi)^2.
\end{align*}
This is the previous case though. Hence, $\gamma(s,\pi_1^{0}\times\tau, \psi)
    =\gamma(s,\pi_2^{0}\times\tau, \psi).$
    
Therefore, in any case, we have shown that $\gamma(s,\pi_1^{0}\times\tau, \psi)
    =\gamma(s,\pi_2^{0}\times\tau, \psi).$ 
    Thus, from \eqref{eqn gamma decomp}, we have
    \begin{align*}
        &\prod_{i=1}^{r_1} \gamma(s+z^i_1,\tau^i_1\times\tau, \psi)\prod_{i=1}^{r_1} \gamma(s-z^i_1,(\tau^i_1)^*\times\tau, \psi) \\
        =&\prod_{i=1}^{r_2} \gamma(s+z^i_2,\tau^i_2\times\tau, \psi)\prod_{i=1}^{r_2} \gamma(s-z^i_2,(\tau^i_2)^*\times\tau, \psi).
    \end{align*}
By \cite[Corollary 2.10]{JNS15}, it follows that $r_1=r_2,$ $z_1^i=z_2^i,$ and $\tau_1^i=\tau_2^i$ for any $i=1,\dots,r_1$ (possibly after a permutation of the indices). Furthermore, by Theorem \ref{converse thm intro}, we have $\pi_1^0=\pi_2^0$ or $\pi_1^0=c\cdot\pi_2^0.$ Since $\pi'$ and $c\cdot\pi'$ are the unique generic constituents of
$$
\tau^1_2\nu^{z^1_2}\times\cdots\times\tau^m_2\nu^{z^m_2}\rtimes\pi^0_2,
$$
and
$$
\tau^1_2\nu^{z^1_2}\times\cdots\times\tau^m_2\nu^{z^m_2}\rtimes c\cdot\pi^0_2,
$$
respectively,
it follows that $\pi\cong\pi'$ or $\pi\cong c\cdot\pi'$ and hence proves the theorem.
\end{proof}

\section{Partial Bessel functions}\label{Partial Bessel functions}
In this section, we review Howe vectors and partial Bessel functions. Their analogues played a crucial role in the proof of the local converse theorems considered in \cite{Z18, Z19}. Their properties again play an important role in the manipulations of the zeta integrals in this paper. The primary reference for much of the material on Howe vectors and partial Bessel functions can be found in \cite{B95}.

\subsection{Howe vectors}\label{Howe vectors}
Let $\omega$ be a character of the center, $Z$, of $\SO_{2l}(F)$ and $C_c^\infty(\SO_{2l},\omega)$ be the space of smooth compactly supported functions $f$ satisfying $f(zg)=\omega(z)f(g)$ for any $z\in Z$ and $g\in\SO_{2l}(F).$ Recall that we fixed $\psi$ to be a nontrivial unramified additive character of $F$ and used it to define a character, also denoted by $\psi$, on $U_{\SO_{2l}}.$ Let $C^\infty(\SO_{2l},\omega,\psi)$ be the space of functions $W$ satisfying $W(zg)=\omega(z)W(g)$ and $W(ug)=\psi(u)W(g)$ for any $z\in Z,$ $u\in U_{\SO_{2l}}$, and $g\in\SO_{2l}(F),$ and furthermore that there exists an open compact subgroup $K$ of $\SO_{2l}(F)$ for which $W$ is invariant under right translation.

Let $m>0$ be an integer and define congruence subgroups 
$$
K_m=(I_{2l}+\mathrm{Mat}_{2l\times 2l}(\mathfrak{p}^m))\cap \SO_{2l}(F).
$$
We define a character $\tau_m$ on $K_m$ by setting 
$$
\tau_m(k)=
   \psi\left(\varpi^{-2m} \left(\sum_{i=1}^{l-2} k_{i,i+1}+\frac{1}{4}k_{l-1,l}-\frac{1}{2} k_{l-1,l+1}  \right)\right),
$$
where $k=(k_{i,j})_{i,j=1}^{2l}\in K_m.$
Set
$$
e_m=\mathrm{diag}(\varpi^{-2m(l-1)},\dots,\varpi^{-2m},1,1,\varpi^{2m},\dots,\varpi^{2m(l-1)}).
$$
Let $H_m=e_m K_m e_m\inv$ and define a character $\psi_m$ on $H_m$ via $\psi_m(h)=\tau_m(e_m\inv h e_m)$ for $h\in H_m.$ Let $U_m=U_{\SO_{2l}}\cap H_m.$ Note that $U_{\SO_{2l}}=\cup_{m\geq 1} U_m$ and the character $\psi$ on $U_{\SO_{2l}}$ agrees with $\psi_m$ upon restriction to $U_m.$

Let $\alpha$ be any root of $\SO_{2l}$ and fix an isomorphism $\mathbf{x}_\alpha: F \rightarrow U_\alpha$ where $U_\alpha$ is the root space of $\alpha.$ Let $U_{\alpha,m}=H_m\cap U_\alpha.$

\begin{lemma}[{\cite[p. 18]{B95}}]\label{height}
Let $\alpha$ be a positive root of $\SO_{2l}.$ Then
$$
U_{\alpha,m}=\{\mathbf{x}_\alpha(x) \, | \, x\in\mathfrak{p}^{-(2\mathrm{ht}(\alpha)-1)m}\},
$$
and
$$
U_{-\alpha,m}=\{\mathbf{x}_{-\alpha}(x) \, | \, x\in\mathfrak{p}^{(2\mathrm{ht}(\alpha)+1)m}\},
$$
where $\mathrm{ht}(\alpha)$ denotes the height of $\alpha.$
\end{lemma}

Let $W\in C^\infty(\SO_{2l},\omega,\psi)$ with $W(I_{2l})=1.$ We define 
$$
W_m(g)=\frac{1}{\mathrm{Vol}(U_m)}\int_{U_m}\psi_m\inv(u) W(gu)du.
$$
Let $C=C(W)$ be the maximal positive integer such that $W$ is invariant under right translation by $K_C.$ If $m\geq C,$ then $W_m$ is called a Howe vector.

\begin{lemma}[{\cite[Lemma 3.2]{B95}}]\label{partialBesselprop}
We have
\begin{enumerate}
    \item $W_m(I_{2l})=1,$
    \item if $m\geq C,$ then  $W_m(gh)=\psi_m(h)W_m(g)$ for any $h\in H_m$ and $g\in\SO_{2l}(F).$
\end{enumerate}
\end{lemma}

Hence, for $m\geq C,$ we have
\begin{equation}\label{Besseleqn}
    W_m(ugh)=\psi(u)\psi_m(h)W_m(g),
\end{equation}
for any $u\in U_{\SO_{2l}}$, $h\in H_m$, and $g\in\SO_{2l}(F).$ For $m\geq C$, $W_m$ is called a {\it partial Bessel function}.

Let $f\in C_c^\infty(\SO_{2l},\omega).$ Define $W^f\in C^\infty(\SO_{2l},\omega,\psi)$ by
$$
W^f(g)=\int_{U_{\SO_{2l}}} \psi\inv(u)f(ug)du,
$$
for $g\in\SO_{2l}(F).$ Since $f$ is locally constant and compactly supported, there exists a positive integer $C=C(f)$ such that $W^f$ is invariant under right translation by $K_C.$ Choose $f\in C_c^\infty(\SO_{2l},\omega)$ such that $W^f(I_{2l})=1.$ Then, for $m>C$, we  consider the corresponding partial Bessel function given by
$$
B_m(g;f):=(W^f)_m(g)=\frac{1}{\mathrm{Vol}(U_m)}\int_{ U_{\SO_{2l}}\times U_m}\psi\inv(u)\psi_m\inv(u') f(ugu')dudu'.
$$

\subsection{Conjugate of partial Bessel Function}
Let $\psi_c$ be the character on $U_{\SO_{2l}}$ defined by $\psi_c(u)=\psi(cuc).$ Recall that we fixed a nonzero $\Gamma\in\mathrm{Hom}_{U_{\SO_{2l}}}(\pi,\psi).$ Then $\Gamma\in\mathrm{Hom}_{U_{\SO_{2l}}}(c\cdot\pi,\psi_c)$ is also nonzero. Notice that $c\cdot\pi$ is also $\psi$-generic. In this section, we relate the partial Bessel functions of $\pi$ and $c\cdot\pi$.

Let $$\tilde{t}=\mathrm{diag}(I_{l-1},\frac{-1}{2},-2,I_{l-1})\in T_{\SO_{2l}}.$$
Then, $\psi_c(\tilde{t}\inv u\tilde{t})=\psi(u)$ for any $u\in U_{\SO_{2l}}.$ We also have $c K_m c = K_m,$ $c U_m c = U_m,$ and $\psi_m(c\tilde{t}\inv u \tilde{t}c)=\psi_m(u).$ Given $W\in C^\infty(\SO_{2l},\omega,\psi)$, define $W^c\in C^\infty(\SO_{2l},\omega,\psi)$ by
$W^c(g)=W(c\tilde{t}\inv g \tilde{t}c).$ If $W(I_{2l})=1,$ we also define $W_m^c=(W^c)_m.$ For $f\in C_c^\infty(\SO_{2l},\omega)$, we set $W^{f,c}=(W^f)^c$ and if $W^f(I_{2l})=1,$ then we set $B_m^c(g;f)= (W^{f,c})_m.$

\begin{lemma}
Suppose $W\in C^\infty(\SO_{2l},\omega,\psi)$ and $W(I_{2l})=1.$ Then,
\begin{enumerate}
    \item $W^c(I_{2l})=1,$
    \item $W^c_m=(W_m)^c,$
    \item $C(W)=C(W^c),$
    \item for $m\geq C,$ we have
$$
W_m^c(ugh)=\psi(u)\psi_m(h)W_m^c(g),
$$
for any $u\in U_{\SO_{2l}}$, $h\in H_m$, and $g\in\SO_{2l}(F).$
\end{enumerate}
\end{lemma}

\begin{proof}
$(1)$ is immediate from the definition. By definition and since $c\tilde{t}\inv U_m\tilde{t} c = U_m$, we obtain $(2).$  $(3)$ follows from the fact $c K_m c = K_m.$ Finally, $(4)$ is readily checked from the previous results and definitions. \end{proof}

Thus, we can see that given a partial Bessel function, $W_m$ or $B_m,$ we may construct another partial Bessel function via the map $W_m\mapsto W_m^c$ or $B_m\mapsto B_m^c.$ In fact, this map sends partial Bessel functions in the Whittaker model of $\pi$ to partial Bessel functions in the Whittaker model of $c\cdot\pi.$ Consequently, by the uniqueness of Whittaker models, this map is the identity map on the Whittaker model of $\pi$ if and only if $\pi\cong c\cdot\pi.$

\subsection{Bessel Support}\label{Besselsupp}

Let $W(\SO_{2l})$ be the Weyl group of $\SO_{2l}$ and let $\Delta(\SO_{2l})$ be the set of simple roots. We say that a Weyl element $w\in W(\SO_{2l})$ supports partial Bessel functions if for any $\alpha\in\Delta(\SO_{2l})$, $w\alpha$ is either negative or simple. Let B($\SO_{2l}$) denote the set of Weyl elements which support partial Bessel functions. B($\SO_{2l}$) is called the Bessel support. Then the partial Bessel functions which we are interested in vanish on the Bruhat cells corresponding to Weyl elements outside of the Bessel support as in the following lemma. 

\begin{lemma}
Let $f\in C_c^\infty(\SO_{2l},\omega)$ and $w\notin \mathrm{B}(\SO_{2l}).$ Then there exists an integer $m_w$ (depending on $f$ and $w$) such that for any $b_1,b_2\in B_{\SO_{2l}},$ $B_m(b_1wb_2,f)=0$ for any $m\geq m_w.$ In particular, if we take $m$ to be an integer larger than $\max\{m_w , | , w\not\in \mathrm{B}(\SO_{2l})\},$ then $B_m(g,f)$ vanishes on the Bruhat cell corresponding to any Weyl element outside of the Bessel support.
\end{lemma}

\begin{proof}
Note that $B_{\SO_{2l}}wB_{\SO_{2l}}=B_{\SO_{2l}}wU_{\SO_{2l}}.$ Since $f$ is compactly supported, there exist compact subsets $B'\subset B_{\SO_{2l}}$ and $U'\subseteq U_{\SO_{2l}}$ such that if  $f(bwu)\neq 0$ then $b\in B'$ and $u\in U'.$   We may choose $m$ large enough such that $U'\subseteq U_m.$ Let $u',u\in U_{\SO_{2l}}$ and $u''\in U_m.$ If $f(u'bwuu'')\neq 0$, then $uu''\in U_m$ and hence $u\in U_m.$ It follows that if $B_m(bwu,f)\neq 0$, then $b\in B_{\SO_{2l}}$ and $u\in U_m.$ From Lemma \ref{partialBesselprop}, after possibly increasing $m$, it is enough to show that $B_m(tw,f)=0$ for any $t\in T_{\SO_{2l}}.$

Let $\alpha$ be a simple root for which $w\alpha$ is positive, but not simple. Recall that we have an isomorphism $\mathbf{x}_\alpha: F \rightarrow U_\alpha$ where $U_\alpha$ is the root space of $\alpha.$ Let $s\in F$ such that $\mathbf{x}_\alpha(s)\in U'.$ Since $U'\subseteq U_m$, from Lemma \ref{partialBesselprop},
$
B_m(tw\mathbf{x}_\alpha(s),f)=B_m(tw,f)\psi(\mathbf{x}_\alpha(s)).
$
Also,
$$
B_m(tw\mathbf{x}_\alpha(s),f)=B_m(t\mathbf{x}_{w\alpha}(s)w,f)=B_m(\mathbf{x}_{w\alpha}(s)tw,f)=B_m(\mathbf{x}_{w\alpha}(w\alpha(t)s)tw,f).
$$
Since $w\alpha$ is positive, but not simple, $\psi(\mathbf{x}_{w\alpha}(w\alpha(t)s))=1.$ From Lemma \ref{partialBesselprop}, we thus obtain that $B_m(tw,f)\psi(\mathbf{x}_\alpha(s))=B_m(tw,f).$ By choosing $s$ such that $\psi(\mathbf{x}_\alpha(s))\neq 0,$ we obtain that $B_m(tw,f)=0$. This completes the proof of the lemma.
\end{proof}

Let $\theta_w=\{\alpha\in\Delta(\SO_{2l}) \, | \, w \alpha\in\Phi^+(\SO_{2l})\}.$ The assignment $w\mapsto \theta_w$ gives a bijection from $\mathrm{B}(\SO_{2l})$ to the power set of $\Delta(\SO_{2l})$ which we denote by $\mathcal{P}(\Delta(\SO_{2l}))$. The Weyl elements for which $\theta_w$ is a singleton play a special role in a partition of the Bessel support. We define them next.

Let $n<l$ and recall the definition of $w_n\in\SO_{2n+1}$ used in defining the intertwining operators. Let 
$$
\hat{w}_n=\left\{\begin{array}{cc}
  \left(\begin{matrix}
  I_{l-n-1} & & & & & \\
   & & & & I_n & \\
   & & & 1 &  & & \\
   & & 1 &  & & & \\
   & I_n & & & & & \\
   & & & & & & I_{l-n-1} \\
  \end{matrix}\right)  & \mathrm{if} \, n \, \mathrm{is} \,  \mathrm{odd,} \\
   \left(\begin{matrix}
  I_{l-n-1} & & & & & \\
   & & & & I_n & \\
   & & 1 &  &  & & \\
   & &  & 1 & & & \\
   & I_n & & & & & \\
   & & & & & & I_{l-n-1} \\
  \end{matrix}\right)  & \mathrm{if} \, n \, \mathrm{is} \,  \mathrm{even.} \\
\end{array}\right. 
$$
The image of $w_n$ in $\SO_{2l}$ is $t_n^\prime \hat{w}_n$ where $t_n^\prime=\tilde{t}$ if $n$ is odd and $t_n^\prime=I_{2l}$ if $n$ is even. Note that if $n$ is odd, then $c w_n c=\tilde{t}\inv\hat{w}_n$ and if $n$ is even then $c w_n  c  = \hat{w}_n$ (here we are realizing $cw_n c\in \SO_{2l}$ via the embedding). Recall the definition of $w^{l,n}$ in the zeta integrals. Let 
$$
\tilde{w}_n=w^{l,n}\hat{w}_n(w^{l,n})\inv=\left(\begin{matrix}
& & & & I_n \\
& I_{l-n-1} & & & \\
& & J_2^n & & \\
& & & I_{l-n-1} & \\
I_n & & & & 
\end{matrix}\right).
$$
Let $\alpha_i$ be the simple roots of $\SO_{2l}$ where $\alpha_i (t)=\frac{t_i}{t_{i+1}}$ for $i\leq l-1$ and $\alpha_l(t)=t_{l-1}t_l$ where $t=\mathrm{diag}(t_1,\dots,t_l,t_l\inv,\dots,t_1).$ We also let $\Phi(\SO_{2l})$, respectively $\Phi^+(\SO_{2l})$, be the set of roots, respectively positive roots, of $\SO_{2l}.$ 
Then, for $n\leq l-2$, we have
\begin{align*}
    \tilde{w}_n\alpha_i & = \alpha_{n-i} \, \, \mathrm{for} \, \, 1\leq i \leq n-1,\\
    \tilde{w}_n\alpha_n (t) & = t_1\inv t_{n+1}\inv,\\
    \tilde{w}_n\alpha_i & =\alpha_i \, \, \mathrm{for} \, \, n+1\leq i \leq l-2,\\
    \tilde{w}_n \alpha_{l-1} & =\left\{\begin{array}{cc}
\alpha_l & \mathrm{if} \, n \, \mathrm{is} \,  \mathrm{odd,} \\
   \alpha_{l-1}  & \mathrm{if} \, n \, \mathrm{is} \,  \mathrm{even,} \\
\end{array}\right.\\
\tilde{w}_n \alpha_{l} & =\left\{\begin{array}{cc}
\alpha_{l-1} & \mathrm{if} \, n \, \mathrm{is} \,  \mathrm{odd,} \\
   \alpha_{l}  & \mathrm{if} \, n \, \mathrm{is} \,  \mathrm{even.} \\
\end{array}\right.
\end{align*}
Note that $\tilde{w}_{l-1}$ sends $\alpha_{l-1}$ and $\alpha_l$ to negative roots and the rest to simple roots. More precisely,
\begin{align*}
    \tilde{w}_{l-1}\alpha_i &= \alpha_{l-1-i} \, \, \mathrm{for} \, \, 1\leq i \leq l-2,\\
    \tilde{w}_{l-1} \alpha_{l-1}(t)&=\left\{\begin{array}{cc}
  t_1\inv t_l\inv & \mathrm{if} \, l-1 \, \mathrm{is} \,  \mathrm{odd,} \\
   t_1\inv t_l  & \mathrm{if} \, l-1 \, \mathrm{is} \,  \mathrm{even},
   \end{array}\right.\\
   \tilde{w}_{l-1} \alpha_{l}(t)&=\left\{\begin{array}{cc}
 t_1\inv t_l & \mathrm{if} \, l-1 \, \mathrm{is} \,  \mathrm{odd,} \\
    t_1\inv t_l\inv  & \mathrm{if} \, l-1 \, \mathrm{is} \,  \mathrm{even.} 
\end{array}\right.
\end{align*}
In summary, for $n\leq l-2,$ $\theta_{\tilde{w}_n}=\Delta(\SO_{2l})\backslash \{\alpha_n\}$ and $\theta_{\tilde{w}_{l-1}}=\Delta(\SO_{2l})\backslash \{\alpha_{l-1}, \alpha_l \}.$

Next, we define the Weyl elements which correspond to the sets $\Delta(\SO_{2l})\backslash \{\alpha_{l-1}\}$ and $\Delta(\SO_{2l})\backslash \{ \alpha_l \}$.
If $l$ is even, let
$$
\tilde{w^\prime_l}=\left(\begin{matrix}
 & I_l \\
 I_l & \\
\end{matrix}\right),
$$
while if $l$ is odd we set 
$$
\tilde{w^\prime_l}=\left(\begin{matrix}
 & & I_{l-1} & \\
 1 & & & \\
 & & & 1 \\
& I_{l-1} & & \\
\end{matrix}\right).
$$
We let $w_{long}$ be the long Weyl element of $\SO_{2l}.$ Then, 

$$
w_{long} =\left\{\begin{array}{cc}
  \left(\begin{matrix}
   &  & J_{l-1} \\
   & I_2 & \\
   J_{l-1} & & 
  \end{matrix}\right)  & \mathrm{if} \, l \, \mathrm{is} \,  \mathrm{odd,} \\
   J_{2l}  & \mathrm{if} \, l \, \mathrm{is} \,  \mathrm{even.} \\
\end{array}\right. 
$$
We have that $$w_{long}\tilde{w^\prime_l}
=\left(\begin{matrix}
J_l & \\
& J_l
\end{matrix}\right),
$$
which is the long Weyl element of the standard Levi subgroup that is isomorphic to $\GL_l.$ This suggests that $\theta_{\tilde{w^\prime_l}}$ should be $\Delta(\SO_{2l})\backslash \{\alpha_{l-1}\}$ or $\Delta(\SO_{2l})\backslash \{ \alpha_l \}.$ We verify this explicitly below. Its conjugation by $c$ gives the other set since $c$ acts on the simple roots by swapping $\alpha_{l-1}$ and $\alpha_l$ and fixing the rest. Set $\tilde{w^\prime_l}^c=c\tilde{w^\prime_l}c.$

We have $\tilde{w^\prime_l} t  \tilde{w^\prime_l}\inv=\mathrm{diag}(t_l\inv,\dots,t_2\inv,t_1^{(-1)^{l+1}},t_1^{(-1)^{l}},t_2,\dots,t_l).$ Hence,
\begin{align*}
    \tilde{w^\prime_l} \alpha_i & =\alpha_{l-i}\, \, \mathrm{for} \, \, 1\leq i \leq l-2,\\
    \tilde{w^\prime_l} \alpha_{l-1}(t) & =\left\{\begin{array}{cc}
t_2\inv t_1\inv & \mathrm{if} \, l \, \mathrm{is} \,  \mathrm{odd,} \\
   \alpha_{1}(t)  & \mathrm{if} \, l \, \mathrm{is} \,  \mathrm{even,} \\
\end{array}\right.\\
\tilde{w^\prime_l} \alpha_{l}(t) & =\left\{\begin{array}{cc}
\alpha_1 (t) & \mathrm{if} \, l \, \mathrm{is} \,  \mathrm{odd,} \\
   t_2\inv t_1\inv  & \mathrm{if} \, l \, \mathrm{is} \,  \mathrm{even.} \\
\end{array}\right.
\end{align*}
Thus, $\theta_{\tilde{w^\prime_l}}=\Delta(\SO_{2l})\backslash\{\alpha_{l-1}\}$ if $l$ is odd and $\theta_{\tilde{w^\prime_l}}=\Delta(\SO_{2l})\backslash\{\alpha_{l}\}$ is even. Similarly, $\theta_{\tilde{w^\prime_l}^c}=\Delta(\SO_{2l})\backslash\{\alpha_{l}\}$ if $l$ is odd and $\theta_{\tilde{w^\prime_l}^c}=\Delta(\SO_{2l})\backslash\{\alpha_{l-1}\}$ is even. 
Hence, we set 
$$
\tilde{w}_l =\left\{\begin{array}{cc}
 \tilde{w^\prime}^c_l  & \mathrm{if} \, l \, \mathrm{is} \,  \mathrm{odd,} \\
   \tilde{w^\prime_l}  & \mathrm{if} \, l \, \mathrm{is} \,  \mathrm{even.} \\
\end{array}\right. 
$$
 and $\tilde{w}_l^c=c\tilde{w}_l c$. Then, $\theta_{\tilde{w}_l}=\Delta(\SO_{2l})\backslash\{\alpha_{l}\}$ and $\theta_{\tilde{w}^c_l}=\Delta(\SO_{2l})\backslash\{\alpha_{l-1}\}.$
 The Weyl elements $\tilde{w}_n$ for $n\leq l$ and $\tilde{w}_l^c$ are used later to define a partition of the Bessel support.

\subsection{Bruhat Order}\label{Bruhat order}
Much of the discussion in this section follows \cite{Z18, Z19}. For $w\in W(\SO_{2l}),$ let $C(w)=B_{\SO_{2l}}wB_{\SO_{2l}}$. We define the Bruhat order on $W(\SO_{2l})$ by $w \leq w'$ if and only if $C(w)\subseteq \ol{C(w')}.$ Then
$$
\ol{C(w')}=\bigsqcup_{
w\leq w'} C(w).
$$
We also let $$\Omega_w = \bigsqcup_{ w'\geq w} C(w').$$ Consequently, $C(w)$ is closed in $\Omega_w.$ Recall that we defined $\psi_n$ to be the coefficients of the generic character $\psi$ in \S\ref{gps and reps}, for $n \leq l$. We let 
$$T_w=\{t\in T_{\SO_{2l}} \, | \, \psi_{n'} \alpha_{n'}(t)=\psi_{n} \, \mathrm{for \, any} \, \alpha_n\in\theta_w \, \mathrm{such \, that}\, w\alpha_{n}=\alpha_{n'} \}.$$

The following Proposition is \cite[Proposition 4.3]{Z18}. There it was proved for symplectic groups, but the proof is general enough to apply directly to our case.

\begin{prop}\label{Omega_w}
\begin{enumerate}
\item[]
    \item If $w,w'\in W(\SO_{2l})$ with $w'>w$ then $\Omega_{w'}$ is an open subset of $\Omega_w.$ In particular, for any $w\in W(\SO_{2l})$, $\Omega_w$ is open in $\Omega_{I_{2l}}=\SO_{2l}.$
    \item Let $P$ be a standard parabolic subgroup of $\SO_{2l}$ and $w\in W(\SO_{2l}).$ Then $PwP\cap\Omega_w$ is closed in $\Omega_w.$
\end{enumerate}
\end{prop}

Suppose that $w\in \mathrm{B}(\SO_{2l}).$ Then, $w_l w$ is the longest Weyl element of some standard Levi subgroup which we denote by $M_w.$ 

\begin{lemma}[{\cite[Proposition 2.1]{CPSSh}}]\label{CPSSh}
Let $w,w'\in \mathrm{B}(\SO_{2l}).$ Then $w'\leq w$ if and only if $M_{w}\subseteq M_{w'}$ if and only if $\theta_{w}\subseteq\theta_{w'}.$
\end{lemma}

The previous lemma allows us to establish some local analogues of the results in \cite{HL22}. We recall the setup. For $a\in\GL_n,$ we define $t_n(a)=\mathrm{diag}(a,I_{2(l-n)},a^*).$ For $n\leq l-2$, let $\mathrm{B}_n(\SO_{2l})$ be the set of $w\in\mathrm{B}(\SO_{2l})$ such that there exists $w^\prime\in W(\GL_n)$ such that $w=t_n(w^\prime)\tilde{w}_n.$ We let $\mathrm{B}_l(\SO_{2l})$ be the set of $w\in\mathrm{B}(\SO_{2l})$ such that there exists $w^\prime\in W(\GL_l)$ such that $w=t_l(w^\prime)\tilde{w}_l$ and $cwc\neq w,$ and let $\mathrm{B}^c_l(\SO_{2l})$ be the set of $cwc$ where $w\in\mathrm{B}_l(\SO_{2l}).$ Finally, we let $\mathrm{B}_{l-1}(\SO_{2l})$ be the set of $w\in\mathrm{B}(\SO_{2l})$ such that there exists $w^\prime\in W(\GL_l)$ such that $w=t_l(w^\prime)\tilde{w}_l$ and $cwc= w.$ By convention, we also define $\mathrm{B}_0(\SO_{2l})=\{I_{2l}\}.$

\begin{prop}[{\cite[Proposition 4.8]{HL22}}]\label{Besselpart} 
The sets $\mathrm{B}_n(\SO_{2l})$ for $n=0,\dots,l$ and $\mathrm{B}^c_l(\SO_{2l})$ form a partition of $\mathrm{B}(\SO_{2l}).$
\end{prop}

Suppose $w,w'\in \mathrm{B}(\SO_{2l})$ with  $w'<w.$ Following \cite{J}, we define the {\it Bessel distance} between $w$ and $w'$ as
$$
d_B(w,w')=\mathrm{max}\{m \, | \, \exists \, w_i\in \mathrm{B}(\SO_{2l}) \, \mathrm{with} \, w=w_m > w_{m-1} > \cdots > w_0=w' \}.
$$ 
From the previous lemma, if $d_B(w,I_{2l})=1,$ then there exists $\alpha\in\Delta(\SO_{2l})$ such that $\theta_w=\Delta(\SO_{2l})\backslash\{\alpha\}.$ Since there are $l$ simple roots, we see that $w$ must be one of $\tilde{w}_n$ for $n\leq l-2$, or $\tilde{w}_l$, or $\tilde{w}_l^c.$ 

For $w_1,w_2\in W(\SO_{2l})$ with $w_1\leq w_2$ we define the closed Bruhat interval as
$$
[w_1,w_2]=\{w\in W(\SO_{2l}) \, | \, w_1\leq w \leq w_2 \}. 
$$
Let $P_{n, 1^{l-n}}$ be the standard parabolic subgroup of $\SO_{2l}$ with Levi subgroup $L_{n,1^{l-n}}\cong \GL_n\times\GL_1^{l-n}.$

\begin{lemma}\label{P_{k, 1^{l-k}}}
For $n\leq l-2,$ the set 
$$
\{w\in W(\SO_{2l}) \, | \, C(w) \subseteq P_{n, 1^{l-n}} \tilde{w}_n P_{n, 1^{l-n}}\}
$$
is equal to the Bruhat interval $[\tilde{w}_n, w_l^{L_{n, 1^{l-n}}}\tilde{w}_n]$ where $w_l^{L_{n, 1^{l-n}}}$ is the long Weyl element in $L_{n, 1^{l-n}}.$
\end{lemma}

\begin{proof}
This proof is similar to \cite[Lemma 4.5]{Z18}; however, we provide it here for completeness.
By \cite[Proposition 2]{BKPST}, the set $$D=\{w\in W(\SO_{2l}) \, | \, C(w) \subseteq P_{n, 1^{l-n}} \tilde{w}_n P_{n, 1^{l-n}}\}
$$ is a Bruhat interval $[w_{\mathrm{min}},w_{\mathrm{max}}].$ The simple roots of $L_{n, 1^{l-n}}$ are precisely $\Delta_{L_{n, 1^{l-n}}}=\{\alpha_1,\dots,\alpha_{n-1}\}.$ Let $W(\Delta_{L_{n, 1^{l-n}}})$ be the Weyl group for $L_{n, 1^{l-n}}.$ Then, by \cite[Corollary 5.20]{BT}, $D$ is the double coset $W(\Delta_{L_{n, 1^{l-n}}})\tilde{w}_n W(\Delta_{L_{n, 1^{l-n}}}).$ We have $\tilde{w}_n=\tilde{w}_n\inv$ and $\tilde{w}_n(\Delta_{L_{n, 1^{l-n}}})>0.$ By \cite[Proposition 1.1.3]{Cas} or \cite[Proposition 2(2)]{BKPST}, $w_{\mathrm{min}}=\tilde{w}_n.$ 

By \cite[Corollary 3]{BKPST}, every $w\in W(\Delta_{L_{n, 1^{l-n}}})\tilde{w}_n W(\Delta_{L_{n, 1^{l-n}}})$ can be written uniquely as $w=\tilde{w}_n w'$ where $w'\in W(\Delta_{L_{n, 1^{l-n}}}).$ Furthermore, $l(w)=l(\tilde{w}_n)+l(w')$ where $l$ denotes the length of the Weyl element. Hence, the maximal such $w$ is $w_{\mathrm{max}}=\tilde{w}_n w_l^{L_{n, 1^{l-n}}}= w_l^{L_{n, 1^{l-n}}} \tilde{w}_n.$
\end{proof}

Note that for $n\leq l-2,$ $[\tilde{w}_n, w_l^{L_{n, 1^{l-n}}}\tilde{w}_n]\cap\mathrm{B}(\SO_{2l})=\mathrm{B}_n(\SO_{2l}).$ Hoping for a similar result when $n=l-1$ or $n=l$ turns out to be troublesome.

Let $P_{l-1, 1}$ be the standard parabolic subgroup with Levi subgroup $L_{l-1,1}\cong \GL_{l-1}\times\GL_1.$
\begin{lemma}
The set
$$
\{w\in W(\SO_{2l}) \, | \, C(w) \subseteq P_{l-1, 1} \tilde{w}_{l-1} P_{l-1, 1}\}
$$
is equal to the Bruhat interval $[\tilde{w}_{l-1}, w_l^{L_{l-1, 1}}\tilde{w}_{l-1}]$ where $w_l^{L_{l-1, 1}}$ is the long Weyl element in $L_{l-1, 1}.$
\end{lemma}

\begin{proof}
This proof is similar to that of Lemma \ref{P_{k, 1^{l-k}}} above.
By \cite[Proposition 2]{BKPST}, the set $D=\{w\in W(\SO_{2l}) \, | \, C(w) \subseteq P_{l-1, 1} \tilde{w}_{l-1} P_{l-1, 1}\}
$ is a Bruhat interval $[w_{\mathrm{min}},w_{\mathrm{max}}].$ The simple roots of $L_{l-1, 1}$ are $\Delta_{L_{l-1, 1}}=\{\alpha_1,\dots,\alpha_{l-2}\}.$ Let $W(\Delta_{L_{l-1, 1}})$ be the Weyl group for $L_{l-1, 1}.$ By \cite[Corollary 5.20]{BT}, $D$ is the double coset given by $W(\Delta_{L_{l-1, 1}})\tilde{w}_{l-1} W(\Delta_{L_{l-1, 1}}).$ We have $\tilde{w}_{l-1}=\tilde{w}_{l-1}\inv$ and $\tilde{w}_{l-1}(\Delta_{L_{l-1, 1}})>0.$ Thus, from \cite[Proposition 1.1.3]{Cas} or \cite[Proposition 2(2)]{BKPST}, $w_{\mathrm{min}}=\tilde{w}_{l-1}.$ 

By \cite[Corollary 3]{BKPST}, every $w\in W(\Delta_{L_{l-1, 1}})\tilde{w}_{l-1} W(\Delta_{L_{l-1, 1}})$ can be written uniquely as $w=\tilde{w}_{l-1} w'$ where $w'\in W(\Delta_{L_{l-1, 1}}).$ Furthermore, $l(w)=l(\tilde{w}_{l-1})+l(w')$ where $l$ denotes the length of the Weyl element. Hence, the maximal such $w$ is $w_{\mathrm{max}}=\tilde{w}_l-1 w_l^{L_{l-1, 1}}= w_l^{L_{l-1, 1}} \tilde{w}_{l-1}.$
\end{proof}

Similar to the cases of $n \leq l-2$ above, $[\tilde{w}_{l-1}, w_l^{L_{l-1, 1}}\tilde{w}_{l-1}]\cap\mathrm{B}(\SO_{2l})=\mathrm{B}_{l-1}(\SO_{2l}).$ We might wish for a similar result in the $n=l$ case; however, we shall see below that this is not the case.

\begin{lemma}
The sets
$$
\{w\in W(\SO_{2l}) \, | \, C(w) \subseteq P_{l} \tilde{w}_{l} P_{l}\}
$$
and
$$
\{w\in W(\SO_{2l}) \, | \, C(w) \subseteq P_l^c \tilde{w}_{l}^c P_l^c\}
$$
are equal to the Bruhat intervals $[\tilde{w}_{l}, w_l^{L_{l}}\tilde{w}_{l}]$ and $[\tilde{w}_{l}^c, w_l^{L_{l}^c}\tilde{w}_{l}^c]$, respectively, where $w_l^{L_{l}}$ is the long Weyl element in $L_{l}$ and $w_l^{L_{l}^c}$ is the long Weyl element in $L_{l}^c$.
\end{lemma}

\begin{proof}
The proof is the same as the other cases. Note that the simple roots of $L_l$ are $\{\alpha_1,\dots,\alpha_{l-1}\}$ and the simple roots of $L_l^c$ are $\{\alpha_1,\dots,\alpha_{l-2},\alpha_{l}\}.$
\end{proof}

The following result describes the main obstruction in generalizing Zhang's work \cite{Z18, Z19} directly to the case of split even special orthogonal groups. 

\begin{lemma}\label{obstruction}
$[\tilde{w}_{l}, w_l^{L_{l}}\tilde{w}_{l}]\cap [\tilde{w}_{l}^c, w_l^{L_{l}^c}\tilde{w}_{l}^c]\cap \mathrm{B}(\SO_{2l})=[\tilde{w}_{l-1}, w_l^{L_{l-1, 1}}\tilde{w}_{l-1}]\cap \mathrm{B}(\SO_{2l}).$
\end{lemma}

\begin{proof}
We have
$$[\tilde{w}_{l}, w_l^{L_{l}}\tilde{w}_{l}]\cap\{w\in W(\SO_{2l}) \, | \, cwc\neq w\} \cap \mathrm{B}(\SO_{2l})= \mathrm{B}_l(\SO_{2l})$$ and 
$$[\tilde{w}_{l}^c, w_l^{L_{l}^c}\tilde{w}_{l}^c]\cap\{w\in W(\SO_{2l}) \, | \, cwc\neq w\} \cap \mathrm{B}(\SO_{2l})= \mathrm{B}_l^c(\SO_{2l}).$$ The claim then follows from the discussion right before \cite[Proposition 4.5]{HL22}.
\end{proof}

This intersection being nonempty has several consequences in the proof of the converse theorem (see discussion after Theorem \ref{converse thm intro}). The next theorem gives the analogues of \cite[Lemmas 5.13 and 5.14]{CST}. Similar to the cases considered in \cite{Z18, Z19}, their proof is general enough to apply to our setting.

\begin{thm}[Cogdell-Shahidi-Tsai]\label{CST}
\begin{enumerate}
\item[]
    \item Let $w\in \mathrm{B}(\SO_{2l})$ and $f\in C_c^\infty(\Omega_w,\omega).$ Suppose $B_m(tw,f)=0$ for any $t\in T_w$. Then, there exists $f'\in C_c^\infty(\Omega_w\backslash C(w),\omega)$, depending only on $f$, such that for $m$ large enough and also only depending on $f$, $B_m(g,f)=B_m(g,f')$ for any $g\in\SO_{2l}(F).$
    \item Let $w\in \mathrm{B}(\SO_{2l})$ and $\Omega_{w,0},\Omega_{w,1}\subseteq\Omega_{w}$ be two open subsets which are $U_{\SO_{2l}}\times U_{\SO_{2l}}$ and $T_{\SO_{2l}}$-invariant such that $\Omega_{w,0}\subseteq\Omega_{w,1}$ and $\Omega_{w,1}\backslash\Omega_{w,0}$ is a union of Bruhat cells $C(w')$ where $w'\notin \mathrm{B}(\SO_{2l}).$ Then, for any $f_1\in C_c^\infty(\Omega_{w,1},\omega),$ there exists $f'\in C_c^\infty(\Omega_{w,0},\omega)$ such that for $m$ sufficiently large and depending only on $f$, $B_m(g,f_1)=B_m(g,f')$ for any $g\in\SO_{2l}(F).$
\end{enumerate}
\end{thm}
The following corollary is an analogue of \cite[Proposition 5.3]{CST}.

\begin{cor}\label{C(0)}
Let $f,f'\in C_c^\infty(\SO_{2l},\omega)$ such that $W^f(I_{2l})=W^{f'}(I_{2l})=1.$ Then, for any $1\leq n \leq l-2$ there exists $f_{\tilde{w}_n}\in C_c^\infty(\Omega_{\tilde{w}_n},\omega)$ and $f_{\tilde{w}_l}\in C_c^\infty(\Omega_{\tilde{w}_l},\omega)$ and $f_{\tilde{w}_l^c}\in C_c^\infty(\Omega_{\tilde{w}_l^c},\omega)$, depending both on $f$ and $f'$,
such that for $m$ large enough 
$$
B_m(g,f)-B_m(g,f')=\sum_{n=1}^{l-2} B_m(g,f_{\tilde{w}_n}) + B_m(g,f_{\tilde{w}_l}) + B_m(g,f_{\tilde{w}_l^c}).
$$
\end{cor}

\begin{proof}
The proof is similar to \cite[Corollary 4.7]{Z18}. We omit it here. 
\end{proof}

\subsection{Outline of the proof of Theorem \ref{converse thm intro}}\label{outline of the proof}

In this section, we outline the strategy of the proof of Theorem \ref{converse thm intro}. Let $\mathcal{C}(0)$ denote the condition that $\pi$ and $\pi'$ have the same central character $\omega$. We always assume $\mathcal{C}(0).$ For $n\geq 1,$ define the condition $\mathcal{C}(n)$ by $\gamma(s,\pi\times\tau,\psi)=\gamma(s,\pi'\times\tau,\psi)$, as functions of the complex variable $s$, for all irreducible generic representations $\tau$ of $\GL_k$ with $1\leq k \leq n$ along with $\mathcal{C}(0).$ Note that $\mathcal{C}(n)$ implies $\mathcal{C}(n-1).$ Since $\pi$ is supercuspidal, its matrix coefficients, which we denote by $\mathcal{M}(\pi)$, are compactly supported. That is, $\mathcal{M}(\pi)\subseteq C_c^\infty(\SO_{2l},\omega).$ Since $\pi$ is generic, the linear functional $\mathcal{M}(\pi)\rightarrow C^\infty(\SO_{2l},\omega,\psi)$ given by $f\mapsto W^f$ is nonzero.

Fix $f\in\mathcal{M}(\pi)$ and $f'\in\mathcal{M}(\pi')$ such that $W^f(I_{2l})=W^{f'}(I_{2l})=1.$ Since $\pi$ and $\pi'$ have satisfy $\mathcal{C}(0),$ by Corollary \ref{C(0)}, there exists $f_{\tilde{w}_n}\in C_c^\infty(\Omega_{\tilde{w}_n},\omega)$ for $1\leq n \leq l-2$ and $f_{\tilde{w}_l}\in C_c^\infty(\Omega_{\tilde{w}_l},\omega)$ and $f_{\tilde{w}_l^c}\in C_c^\infty(\Omega_{\tilde{w}_l^c},\omega)$ such that for $m$ large enough 
\begin{equation}\label{Bessel Difference}
B_m(g,f)-B_m(g,f')=\sum_{n=1}^{l-2} B_m(g,f_{\tilde{w}_n}) + B_m(g,f_{\tilde{w}_l}) + B_m(g,f_{\tilde{w}_l^c}).
\end{equation}

As in \cite{Z18,Z19}, we shall show that the condition $\mathcal{C}(k)$ implies
$$
B_m(g,f)-B_m(g,f')=\sum_{n=k+1}^{l-2} B_m(g,f_{\tilde{w}_n}) + B_m(g,f_{\tilde{w}_l}) + B_m(g,f_{\tilde{w}_l^c}).
$$
for $k\leq l-2.$ Thus $\mathcal{C}(l-2)$ gives
$$
B_m(g,f)-B_m(g,f')= B_m(g,f_{\tilde{w}_l}) + B_m(g,f_{\tilde{w}_l^c}).
$$
However, unlike the symplectic and unitary cases,  $\Omega_{\tilde{w}_l}\cap\Omega_{\tilde{w}_l^c}\neq\emptyset.$ Indeed, from Lemma \ref{obstruction}, we  see that 
$$\Omega_{\tilde{w}_l}\cap\Omega_{\tilde{w}_l^c}=\bigsqcup_{w'\geq \tilde{w}_l   , \, cw'c=w'} C(w')=\Omega_{\tilde{w}_{l-1}}.$$ In particular, the Bruhat cell of the long Weyl element lies in this intersection. We  find that $\mathcal{C}(l-1)$ determines $$B_m( t_{l-1}(a)t_{l-1}' \tilde{w}_{l-1} ,f_{\tilde{w}_l}) + B_m( t_{l-1}(a)t_{l-1}' \tilde{w}_{l-1} ,f_{\tilde{w}_l^c})=0,
$$ for $a\in\GL_{l-1}$. Hence, we know this sum on a subset of the intersection (this is due to the embedding $\SO_{2l-1}$ into $\SO_{2l}$ not capturing the full Bruhat cell). Consequently, we must determine the remainder of the intersection, along with $\Omega_{\tilde{w}_l}\backslash\Omega_{\tilde{w}_{l-1}}$ and $\Omega_{\tilde{w}_l^c}\backslash\Omega_{\tilde{w}_{l-1}}$ using $\mathcal{C}(l).$ A key step in establishing these results is to use Proposition \ref{JS Prop} in \S \ref{proof of the local converse theorem}. To do this for $\mathcal{C}(l)$, we have to use conjugation by $c\tilde{t}\inv$ to relate these sets. This results in an integrand involving
\begin{align*}
&\,B_m(g,f)-B_m(g,f')+B_m^c(g,f)-B_m^c(g,f') \\
= &\,B_m(g,f_{\tilde{w}_l}) + B_m(g,f_{\tilde{w}_l^c})+B_m^c(g,f_{\tilde{w}_l}) + B_m^c(g,f_{\tilde{w}_l^c}).
\end{align*}
Consequently, we deduce that
$$
B_m(g,f)-B_m(g,f')+B_m^c(g,f)-B_m^c(g,f')=0,$$ hence, by the uniqueness of Whittaker models, we obtain that $\pi\cong\pi'$ or $\pi\cong c\cdot\pi'$. This completes the outline of the proof of Theorem \ref{converse thm intro}.

\section{Sections of induced representations}\label{sections}
Let $n\leq l$ and $(\tau, V_\tau)$ be an irreducible $\psi\inv$-generic representation of $\GL_n$. In this section, we construct a section in $I(\tau,s,\psi\inv)$ which is used in calculating the zeta integrals. Let $\ol{V}_n$ be the unipotent radical of the opposite parabolic to $Q_n$. That is,
$$
\ol{V}_n=\left\{\left(\begin{matrix}
I_n & & \\
* & 1 & \\
* & * & I_n
\end{matrix}\right)\in\SO_{2n+1}\right\}.
$$
Let $i$ be a positive integer and recall the definition of $H_i\subseteq\SO_{2l}$ in  \S\ref{Howe vectors}. For $n<l,$ set $\ol{V}_{n,i}=w^{l,n}\ol{V}_n(w^{l,n})\inv\cap H_i$ where we realize $\ol{V}_n$ through the embedding of $\SO_{2n+1}$ into $\SO_{2l}.$ 

For a positive integer $m$, let $K_m^{\GL_n}$ be the congruence subgroups of $\GL_n.$ That is, $K_m^{\GL_n}=I_{n}+\mathrm{Mat}_{n\times n}(\mathfrak{p}^m).$ Following \cite[\S 3.2]{B95}, we define an open compact subgroup of $\SO_{2l+1}$ by $$H_m^{\SO_{2l+1}}=e'_m(K_m^{\GL_{2l+1}}\cap\SO_{2l+1})(e'_m)\inv$$ where $e'_m=\mathrm{diag}(\varpi^{-2ml},\dots,\varpi^{2ml})\in\SO_{2l+1}.$ If $n=l,$ we set $\ol{V}_{l,i}=\ol{V}_l\cap H_i^{\SO_{2l+1}}.$

Let $D$ be some open, compact subgroup of $V_n$. For $x\in D$, we let 
$$
S(x,i)=\{\ol{y}\in\ol{V}_n \, | \, \ol{y}x\in Q_n \ol{V}_{n,i}\}.
$$ 
The following two lemmas follow from an analogous argument as \cite[Lemma 5.1]{Z18}. One can also find similar statements in \cite[\S 4]{B95}.

\begin{lemma}\label{5.1}
For any positive integer $m$, there exists an integer $i_1=i_1(D,m)$ such that for any $i\geq i_1, x\in D,$ and $\ol{y}\in S(x,i)$ we have 
$$
\ol{y}x=vl_n(a)\ol{y}_0
$$ where $v\in V_n, a\in K_m^{\GL_n}$, and $\ol{y}_0\in\ol{V}_{n,i}.$
\end{lemma}

\begin{lemma}
There exists an integer $i_2=i_2(D)$ such that $S(x,i)=\ol{V}_{n,i}$ for any $x\in D$ and $i\geq i_2.$
\end{lemma}

For $v\in V_\tau$ define a function $f_s^{i,v}:\SO_{2n+1}\rightarrow V_\tau$ by
$$
f_s^{i,v}(g)=\left\{\begin{array}{cc}
\frac{|\mathrm{det}a|^{s-\frac{1}{2}}\tau(a)v}{W_v(I_n)} & \, \mathrm{if} \, g=u'l_n(a)\ol{u} \, \mathrm{where} \, u'\in V_n, a\in\GL_n, \ol{u}\in \ol{V}_{n,i}, \\
0 & \mathrm{otherwise.}
\end{array}\right.
$$

\begin{lemma}\label{5.3}
For any $v\in V_\tau,$ there exists an integer $i_0(v)$ such that for any $i\geq i_0(v)$, $f_s^{i,v}\in I(\tau,s).$
\end{lemma}

\begin{proof}
The proof is similar to \cite[Lemma 5.2]{Z18}.
\end{proof}

Recall that in \S 2 we fixed a nonzero Whittaker functional $\Lambda_\tau \in \mathrm{Hom}_{U_{\GL_{n}}}(\tau,\psi^{-1}).$ Define a function $\xi_s^{i,v}:\SO_{2n+1}\times\GL_n\rightarrow\mathbb{C}$ by $\xi_s^{i,v}(g,a)=\Lambda_\tau(\tau(a)f^{i,v}_s(g)).$ Then, $\xi_s^{i,v}\in I(\tau,s,\psi\inv)$ and
\begin{equation}\label{xi nonintertwined}
\xi_s^{i,v}(g, I_n)=\left\{\begin{array}{cc}
\frac{|\mathrm{det}a|^{s-\frac{1}{2}}W_v(a)}{W_v(I_n)} & \, \mathrm{if} \, g=u'l_n(a)\ol{u} \,\, \mathrm{where} \, u'\in V_n, a\in\GL_n, \ol{u}\in \ol{V}_{n,i}, \\
0 & \mathrm{otherwise.}
\end{array}\right.
\end{equation}
Recall that $W_v(a)=\Lambda_\tau(\tau(a)v)$ is the Whittaker function of $\tau$ attached to $v.$

Next we consider the image of this function under the intertwining operator. Let $\tilde{\xi}_{1-s}^{i,v}=M(\tau,s,\psi\inv)\xi_s^{i,v}\in I(\tau^*,1-s,\psi\inv).$ The following lemma is our analogue of \cite[Lemma 5.3]{Z18}.

\begin{lemma}\label{5.4}
Let $D$ be an open compact subset of $V_n.$ Then there exists a positive integer $i(D,v)\geq i_0(v)$ such that for any $i\geq i(D,v)$, $\tilde{\xi}_{1-s}^{i,v}(w_n x, I_n)=\mathrm{Vol}(\ol{V}_{n,i}).$
\end{lemma}

\begin{proof}
By definition,
$$
\tilde{\xi}_{1-s}^{i,v}(w_n x, I_n)= \int_{N_n} \xi_{s}^{i,v}(w_n u w_n x, d_n)du= \int_{N_n} \xi_{s}^{i,v}(d_n w_n u w_n x, I_n)du.
$$
$\xi_{s}^{i,v}(d_n w_n u w_n x, I_n)\neq 0$ if and only if  $d_n w_n u w_n x\in Q_n \ol{V}_{n,i}$. We have $d_n\in Q_n$ and hence, $d_n w_n u w_n x\in Q_n \ol{V}_{n,i}$ if and only if $d_n w_n u w_n\in \ol{V}_{n,i}= S(x,i).$ By Lemma \ref{5.1}, for $d_n w_n u w_n\in S(x,i),$
$$
d_n w_n u w_n x=u'l_n(a)\ol{y}_0,
$$ for some $u'\in V_n, a\in K_m^{\GL_n}$ which fixes $v$, and $\ol{y}_0\in\ol{V}_{n,i}.$ Since $\tau(a)v=v$ and $|\mathrm{det}(a)|=1,$ we have $|\mathrm{det}(a)|^{s-\frac{1}{2}}W_v(a)=W_v(I_n)$ and
$$
\xi_{s}^{i,v}(d_n w_n u w_n x, I_n)=\left\{\begin{array}{cc}
1 & \, \mathrm{if} \, d_nw_n u w_n\in \ol{V}_{n,i}, \\
0 & \mathrm{otherwise.}
\end{array}\right.
$$
Therefore,
$$
\tilde{\xi}_{1-s}^{i,v}(w_n x, I_n)=\mathrm{Vol}(\ol{V}_{n,i}).
$$
\end{proof}

The above argument also describes the image of $\tilde{\xi}_{1-s}^{i,v}.$ We have
\begin{equation}\label{xi intertwined} \tilde{\xi}_{1-s}^{i,v}(u_1 l_n(a) w_n u_2, I_n)=\left\{\begin{array}{cc}
\mathrm{Vol}(\ol{V}_{n,i})|\mathrm{det}a|^{\frac{1}{2}-s}W_v^*(a) & \, \mathrm{if} \, \, u_1\in V_n, a\in\GL_n, u_2\in D, \\
0 & \mathrm{otherwise.}
\end{array}\right.
\end{equation}

\section{Proof of the local converse theorem}\label{proof of the local converse theorem}

Throughout this section, we let $\pi$ and $\pi'$ be irreducible $\psi$-generic supercuspidal representations of $\SO_{2l}$ with same central character $\omega.$ We also fix matrix coefficients $f\in\mathcal{M}(\pi)$ and $f'\in\mathcal{M}(\pi')$ such that $W^f(I_{2l})=W^{f'}(I_{2l})=1$.

The goal of this section is to prove Theorem \ref{converse thm intro}. To do this, we first follow Zhang's work in the symplectic and unitary cases \cite{Z18, Z19} for twists up to $\GL_{l-2},$ to  obtain that 
$$
B_m(g,f)-B_m(g,f')= B_m(g,f_{\tilde{w}_l}) + B_m(g,f_{\tilde{w}_l^c}).
$$
To compute the right hand side, our work begins to differ from Zhang's. It turns out that twists up to $\GL_{l-1}$  determine the right hand side on certain elements that are fixed by the outer conjugation $c$. This can be seen through explicit computation of the embedding. To obtain the rest, we must also make use of  the twists by $\GL_l$; however, the result instead implies that the right hand side plus its outer conjugation by $c$ is $0$. That is,
$$
B_m(g,f_{\tilde{w}_l}) + B_m(g,f_{\tilde{w}_l^c})+B_m^c(g,f_{\tilde{w}_l}) + B_m^c(g,f_{\tilde{w}_l^c})=0.
$$

\subsection{Twists up to $\GL_{l-2}$}

We consider the twists up to  $\GL_{l-2}.$ Recall from Equation \eqref{Bessel Difference},
$$
B_m(g,f)-B_m(g,f')=\sum_{n=1}^{l-2} B_m(g,f_{\tilde{w}_n}) + B_m(g,f_{\tilde{w}_l}) + B_m(g,f_{\tilde{w}_l^c}).
$$
In this subsection, we prove the following theorem.

\begin{thm}\label{l-2 theorem}
Suppose that $k\leq l-2.$ If $\pi$ and $\pi'$ satisfy the condition $\mathcal{C}(k)$, then
$$
B_m(g,f)-B_m(g,f')=\sum_{n=k+1}^{l-2} B_m(g,f_{\tilde{w}_n}) + B_m(g,f_{\tilde{w}_l}) + B_m(g,f_{\tilde{w}_l^c}),
$$ for any $g\in\SO_{2l}$ and $m$ large enough.
\end{thm}

We shall prove this by induction. The case $k=0$ is already done by Corollary \ref{C(0)}. Next, we shall assume inductive hypothesis. That is, the condition $\mathcal{C}(k-1)$ implies
$$
B_m(g,f)-B_m(g,f')=\sum_{n=k}^{l-2} B_m(g,f_{\tilde{w}_n}) + B_m(g,f_{\tilde{w}_l}) + B_m(g,f_{\tilde{w}_l^c}),
$$ for any $g\in\SO_{2l}$ and $m$ large enough. We then want to show that the condition $\mathcal{C}(k)$ allows us to remove the term $B_m(g,f_{\tilde{w}_k}).$ We begin with some preliminaries.

Let $P_k$ be the standard maximal parabolic subgroup of $\SO_{2l}$ with Levi subgroup $M_k$ isomorphic to $\GL_k\times\SO_{2(l-k)}.$ Let $N_k$ be its unipotent radical. Furthermore, we let
$$
S^+_k=\left\{\left(\begin{matrix}
u_1 & &  \\
    & u_2 & \\
    &   & u_1^*
\end{matrix}\right) \, \left| \, u_1\in U_{\GL_k}, u_2\in U_{\SO_{2(l-k)}}\right\}\right. .
$$
Note that $U_{\SO_{2l}}=N_k S_k^+.$ By \cite[p. 12]{Cas}, the product map gives an isomorphism
$$
P_k \times \{\tilde{w}_k\}\times N_k \rightarrow P_k \tilde{w}_k P_k.
$$
Alternatively, one can check that conjugation by $\tilde{w}_k$ takes $P_k$ to its opposite parabolic. Also, recall that we set $P_{n, 1^{l-n}}$ to be the standard parabolic subgroup with Levi subgroup $L_{n,1^{l-n}}\cong \GL_n\times\GL_1^{l-n}.$

\begin{lemma}\label{6.2}
For $u_0\in N_k\backslash(N_k\cap U_m)$, $u^+\in S_k^+\cap U_m$, we have that $(u^+)\inv u_0 u^+\in N_k\backslash(N_k\cap U_m).$
\end{lemma}

\begin{proof}
This is precisely \cite[Lemma 6.2]{Z18}. The proof there is for symplectic groups, but it is essentially the same for split even special orthogonal groups.
\end{proof}

Let $m$ be a positive integer and recall the definition of $U_m$ from $\S$\ref{Howe vectors}. For $a\in\GL_n$, we let
$$
t_n(a)=\mathrm{diag}(a,I_{2(l-n)},a^*)\in\SO_{2l}.
$$

\begin{lemma}\label{6.3}
\begin{enumerate}
    \item For $m$ large enough and depending only on $f$ and $f'$,
    $$
    B_m(t_l(a),f)-B_m(t_l(a),f')=0,
    $$ for any $a\in\GL_l.$
    \item Suppose that $k\leq l-2.$ For any $i\geq k+1$, $P_{k, 1^{l-k}}\tilde{w}_k P_{k, 1^{l-k}}\cap\Omega_{\tilde{w}_i}=\emptyset.$ Thus, 
    $$
    B_m(g,f_{\tilde{w}_i})=0,
    $$
    for any $g\in P_{k, 1^{l-k}}\tilde{w}_k P_{k, 1^{l-k}}$ and $i\geq k+1.$ Hence
    $$
     B_m(g,f)-B_m(g,f')=B_m(g,f_{\tilde{w}_k}),
    $$
    for any $g\in P_{k, 1^{l-k}}\tilde{w}_k P_{k, 1^{l-k}}$ and $m$ large enough.
    \item Let $k\leq l.$ For $m$ large enough depending only of $f_{\tilde{w}_k}$ (and hence only on $f$ and $f'$), we have
    $$
    B_m(t_k(a)\tilde{w}_k u_0, f_{\tilde{w}_k})=0
    $$
    for all $a\in\GL_k$ and $u_0\in N_k\backslash (U_m\cap N_k).$
\end{enumerate}
\end{lemma}

\begin{proof}
\begin{enumerate}
    \item This follows from the proof of \cite[Proposition 4.9]{HL22}.
    \item For contradiction, suppose there exists $w\in \mathrm{B}(\SO_{2l})$ such that $$C(w)\subseteq P_{k, 1^{l-k}}\tilde{w}_k P_{k, 1^{l-k}}\cap\Omega_{\tilde{w}_k}.$$ By Lemma \ref{P_{k, 1^{l-k}}}, $$w\in [\tilde{w}_k, w_l^{L_{k, 1^{l-k}}}\tilde{w}_k]\cap\mathrm{B}(\SO_{2l})=\mathrm{B}_k(\SO_{2l}).$$ By \cite[Proposition 4.7]{HL22}, we have $$\{\alpha_{k+1},\dots,\alpha_l\}\subseteq\theta_w\subseteq\Delta(\SO_{2l})\backslash\{\alpha_k\}.$$ However, since $w\in\Omega_{\tilde{w}_i}$ $w\geq\tilde{w}_i$ and by Lemma \ref{CPSSh}, $$\theta_w\subseteq\theta_{\tilde{w}_i}=\Delta(\SO_{2l})\backslash\{\alpha_{i}\}.$$
    Since $i\geq k+1$ and $\{\alpha_{k+1},\dots,\alpha_l\}\subseteq\theta_w\subseteq\Delta(\SO_{2l})\backslash\{\alpha_{i}\},$ we have a contradiction.
    \item This proof is the similar to \cite[Lemma 6.3(3)]{Z18} but we include it for completeness. By Proposition \ref{Omega_w}, $P_{k, 1^{l-k}}\tilde{w}_k P_{k, 1^{l-k}}$ is closed in $\Omega_{\tilde{w}_k}.$ Since $f_{\tilde{w}_k}\in C_c^\infty(\Omega_{\tilde{w}_k},\omega)$, there exists compacts subsets $P'\subseteq P_{k, 1^{l-k}}$ and $N'\subseteq N_k$ such that if $f_{\tilde{w}_k}(p\tilde{w}_k n)\neq 0$, then $p\in P'$ and $n\in N'.$ 
    
    Choose $m$ large enough such that $N'\subseteq U_m\cap N_k.$ This choice depends only on $N'$ and hence $f_{\tilde{w}_k}.$ Let $u\in U_m$ and write $u=u^+ u^-$ where $u^+\in S_k^+\cap U_m$ and $u^-\in N_k\cap U_m.$ Let $u_0\in N_k\backslash (U_m\cap N_k).$ By Lemma \ref{6.2}, $u_0':=(u^+)\inv u_0 u^+\in N_k\backslash (U_m\cap N_k).$ Then, for $a\in\GL_k$ and $u'\in U_{\SO_{2l}}$, we have
    \begin{align*}
    f_{\tilde{w}_k}(u't_k(a)\tilde{w}_k u_0 u)
    =f_{\tilde{w}_k}(u't_k(a)\tilde{w}_k u^+ u_0' u^-) =f_{\tilde{w}_k}(u't_k(a)\tilde{w}_k u^+\tilde{w}_k\inv \tilde{w}_k u_0' u^-).
    \end{align*}
    We have $u't_k(a)\tilde{w}_k u^+\tilde{w}_k\inv\in P_{k, 1^{l-k}}.$ Thus if $ f_{\tilde{w}_k}(u't_k(a)\tilde{w}_k u_0 u)\neq 0,$ then $u't_k(a)\tilde{w}_k u^+\tilde{w}_k\inv\in P'$ and $u_0' u^-\in N'\subseteq U_m\cap N_k.$ Since $u^-\in N_k\cap U_m,$ we must have $u_0'\in N_k\cap U_m.$ But, $u_0'\in N_k\backslash (U_m\cap N_k).$ Therefore, for any $u'\in U_{\SO_{2l}}$ and $u\in U_m,$ $f_{\tilde{w}_k}(u't_k(a)\tilde{w}_k u_0 u)=0.$ Finally, by definition, we find that
    $$
    B_m(t_k(a)\tilde{w}_k u_0, f_{\tilde{w}_k})=0,
    $$
    for all $a\in\GL_k$ and $u_0\in N_k\backslash (U_m\cap N_k).$
    \end{enumerate}
\end{proof}

The next proposition is used to connect the vanishing of certain integrals to information about the partial Bessel functions.

\begin{prop}[{\cite[Lemma 2.3]{Cha19}}]\label{JS Prop}
Let $W'$ be a smooth function on $\GL_n$ such that $W'(ug)=\psi(u)W'(g)$ for any $u\in U_{\GL_n}$ and $g\in\GL_n.$ Suppose that for any irreducible generic representation $\tau$ of $\GL_n$ and for any Whittaker function $W\in\mathcal{W}(\tau,\psi\inv),$ the integral 
$$
\int_{U_{\GL_n} \backslash \GL_n} W'(g)W(g)|\mathrm{det}(g)|^{-s-k}=0,
$$
for $\mathrm{Re}(s)<<0.$ Then, $W'=0.$
\end{prop}

The above proposition follows from \cite[Lemma 3.2]{JS}. The proof of the above form follows \cite[Corollary 2.1]{Chen}. See the remark after \cite[Lemma 2.3]{Cha19}.

The next proposition is the main step in proving Theorem \ref{l-2 theorem}.

\begin{prop}\label{l-2 prop}
Let $1\leq k \leq l-2.$ The condition $\mathcal{C}(k)$ implies
$$
B_m(t_k(a)t_k'\tilde{w}_k,f_{\tilde{w}_k})=0,
$$
for any $a\in\GL_k$ and $m$ large enough depending only on $f$ and $f'.$
\end{prop}

\begin{proof}
Let $m$ be large enough such that the induction hypothesis and Lemma \ref{6.3} hold. Let $(\tau, V_\tau)$ be an irreducible $\psi\inv$-generic representation of $\GL_k$ and $v\in V_\tau.$ Since $k\leq l-2 <l,$ we are embedding $\SO_{2k+1}$ into $\SO_{2l}$ as in $\S$\ref{gps and reps}. Let $j:\SO_{2k+1}\rightarrow\SO_{2l}$ denote the embedding. Recall that $Q_k=L_k V_k$ is the standard Siegel parabolic subgroup of $\SO_{2k+1}$ with Levi subgroup $L_k\cong\GL_k.$ Let $V_{k,m}=\{u\in V_k \, | w^{l,k} j(u) (w^{l,k})\inv\in H_m\}.$ $V_{k,m}$ is an open compact subgroup of $V_k.$ Let $i$ be a positive integer with $i\geq\mathrm{max}\{m,i_0(v),i(V_{k,m},v)\}$ (for notation, see Lemma \ref{5.3} for $i_0(v)$ and Lemma \ref{5.4} for $i(V_{k,m},v)$ where we take $D=V_{k,m}$). We consider the sections $\xi_s^{i,v}\in I(\tau,s,\psi\inv)$ and $\tilde{\xi}_{1-s}^{i,v}=M(\tau,s,\psi\inv)\xi_s^{i,v}\in I(\tau^*,1-s,\psi\inv)$ defined in $\S\ref{sections}.$ 

First, we compute the non-intertwined zeta integrals $\Psi(W_m^f,\xi_s^{i,v}).$ We have that $V_k L_k \ol{V}_k$ is dense in $\SO_{2k+1}$ where $\ol{V}_k$ denotes the opposite unipotent radical to $V_k.$ Hence, we can replace the integral over $U_{\SO_{2k+1}}\backslash\SO_{2k+1}$ with an integral over $$U_{\SO_{2k+1}}\backslash V_k L_k \ol{V}_k\cong U_{\GL_k}\backslash\GL_k \times \ol{V}_k,$$
 where $\GL_k$ is realized in $L_k\subseteq\SO_{2k+1}.$ Note for $g=u_1 m_k(a) \ol{u}_2\in U_{\GL_k}\backslash\GL_k \times \ol{V}_k$ we use the quotient Haar measure $dg=|\mathrm{det}(a)|^{-k} d\ol{u}_2 da.$

Thus, we have
\begin{align*}    
&\,\Psi(W_m^f,\xi_s^{i,v})\\
=&\,\int_{U_{\GL_k}\backslash\GL_k}\int_{\ol{V}_k}\left(\int_{r\in R^{l,k}} W_m^f(r w^{l,k} j(l_k(a)\ol{u}) (w^{l,k})\inv)dr\right)   \xi_s^{i,v}(l_k(a)\ol{u}, I_{l})
    |\mathrm{det}(a)|^{-k}d\ol{u}da.
\end{align*}
By Equation \eqref{xi nonintertwined}, we have
\begin{align*}
   &\,\Psi(W_m^f,\xi_s^{i,v})\\
=&\,\int_{U_{\GL_k}\backslash\GL_k}\int_{\ol{V}_{k,i}}\left(\int_{r\in R^{l,k}} W_m^f(r w^{l,k} j(l_k(a)\ol{u}) (w^{l,k})\inv)dr\right)
   |\mathrm{det}(a)|^{s-\frac{1}{2}-k} W_v(a)d\ol{u}da. 
\end{align*}
Furthermore, for $\ol{u}\in\ol{V}_{k,i}$, we have $w^{l,k} j(\ol{u}) (w^{l,k})\inv\in H_m.$ By Proposition \ref{partialBesselprop},
$$
W^f_m(g w^{l,k} j(\ol{u}) (w^{l,k})\inv) = W_m^f(g),
$$
for any $g\in\SO_{2l}(F).$ Thus,
\begin{align*}
  &\,\Psi(W_m^f,\xi_s^{i,v})\\
  =&\,\mathrm{Vol}(\ol{V}_{k,i})\int_{U_{\GL_k}\backslash\GL_k}\left(\int_{r\in R^{l,k}} W_m^f(r w^{l,k} j(l_k(a)) (w^{l,k})\inv)dr\right)
  |\mathrm{det}(a)|^{s-\frac{1}{2}-k} W_v(a)da.  
\end{align*}
Let $q_k(a)=\mathrm{diag}(I_{l-k-1},a,I_2,a^*,I_{l-k-1}).$ We have $j(l_k(a))=q_k(a).$ Thus, 
$$w^{l,k} q_k(a) (w^{l,k})\inv=\mathrm{diag}(a, I_{2l-2k}, a^*).$$
Let $$r_x=\left(\begin{matrix}
I_k & & & & \\
x & I_{l-k-1} & & & \\
& & I_2 & & \\
& & & I_{l-k-1} & \\
& & & x^\prime & I_k
\end{matrix}\right).$$ Then, $$
r_x w^{l,k} q_k(a) (w^{l,k})\inv
=\left(\begin{matrix}
a & & & & \\
xa & I_{l-k-1} & & & \\
& & I_2 & & \\
& & & I_{l-k-1} & \\
& & & x^\prime a^* & a^*
\end{matrix}\right)\in L_l.
$$
Thus, by Lemma \ref{6.3}(1),
$$
W_m^f(r w^{l,k} j(l_k(a)) (w^{l,k})\inv)=W_m^{f'}(r w^{l,k} j(l_k(a)) (w^{l,k})\inv),
$$
and hence 
$$
\Psi(W_m^f,\xi_s^{i,v})=\Psi(W_m^{f'},\xi_s^{i,v}).
$$

By assumption, $\gamma(s,\pi\times\tau,\psi)=\gamma(s,\pi'\times\tau,\psi).$ Since we have that $\Psi(W_m^f,\xi_s^{i,v})=\Psi(W_m^{f'},\xi_s^{i,v}),$ the local functional equation gives
$$\Psi(W_m^f,\tilde{\xi}_{1-s}^{i,v})=\Psi(W_m^{f'},\tilde{\xi}_{1-s}^{i,v}).$$
$V_k L_k w_k V_k$ is dense in $\SO_{2k+1}.$ Hence, we can replace the integral over $U_{\SO_{2k+1}}\backslash
\SO_{2k+1}$ with an integral over $U_{\SO_{2k+1}}\backslash V_k L_k w_k V_k \cong U_{\GL_k}\backslash\GL_k w_k V_k.$ We obtain $\Psi(W_m^f,\tilde{\xi}_{1-s}^{i,v})$ equals
\begin{align*}
&\,\int_{U_{\GL_k}\backslash\GL_k}\int_{V_k}\left(\int_{r\in R^{l,k}} W_m^f(r w^{l,k} j(l_k(a)w_k u) (w^{l,k})\inv)dr\right)  \\
&\,\cdot \tilde{\xi}_{1-s}^{i,v}(l_k(a)w_ku, I_{l})
  |\mathrm{det}(a)|^{-k}duda,  
\end{align*}
and we have a similar expression for $\Psi(W_m^{f'},\tilde{\xi}_{1-s}^{i,v}).$

We have $w^{l,k} j(w_k) (w^{l,k})\inv=t_k'\tilde{w}_k$ and 
$$w^{l,k} j(l_k(a)) (w^{l,k})\inv=\mathrm{diag}(a,I_{2l-2k},a^*)=t_k(a).$$ Again, let
$$r_x=\left(\begin{matrix}
I_k & & & & \\
x & I_{l-k-1} & & & \\
& & I_2 & & \\
& & & I_{l-k-1} & \\
& & & x^\prime & I_k
\end{matrix}\right).$$ Then, $
r_x t_k(a)t_k'=t_k(a)t_k'r_{ax}
$ and 
$$r_{ax}\tilde{w}_k=\tilde{w}_k\left(\begin{matrix}
I_k & & & x'a^* & \\
& I_{l-k-1} & & & xa \\
& & I_2 & & \\
& & & I_{l-k-1} & \\
& & & & I_k
\end{matrix}\right).$$ 
Also, for 
$$u=\left(\begin{matrix}
I_k & y_1   & y \\
    & 1     & y_1' \\
    &       & I_k
\end{matrix}\right)\in V_k,$$ we have 
$$
w^{l,k} j(u) (w^{l,k})\inv= \left(\begin{matrix}
I_k & 0 & * & 0 & * \\
& I_{l-k-1} & 0 & 0 & 0  \\
& & I_2 & 0 & * \\
& & & I_{l-k-1} & 0 \\
& & & & I_k
\end{matrix}\right).
$$
Recall $P_k$ is the standard maximal parabolic subgroup of $\SO_{2l}$ with Levi subgroup $M_k$ isomorphic to $\GL_k\times\SO_{2(l-k)}$ and unipotent radical $N_k.$ We have $$j(a,x,u):=\left(\begin{matrix}
I_k & & & x'a^* & \\
& I_{l-k-1} & & & xa \\
& & I_2 & & \\
& & & I_{l-k-1} & \\
& & & & I_k
\end{matrix}\right)w^{l,k} j(u) (w^{l,k})\inv\in N_k.$$ 
Thus, we have shown that
$$
r_x w^{l,k} j(l_k(a)w_k u) (w^{l,k})\inv 
=t_k(a)t_k'\tilde{w}_k j(a,x,u)\inv\in P_k \tilde{w}_k P_k.
$$
From Lemma \ref{6.3}(2), we obtain
\begin{align*}
&\,W_m^f(t_k(a)t_k'\tilde{w}_k j(a,x,u))-W_m^{f'}(t_k(a)t_k'\tilde{w}_k j(a,x,u)) \\
=&\,B_m(t_k(a)t_k'\tilde{w}_k j(a,x,u),f)-B_m(t_k(a)t_k'\tilde{w}_k j(a,x,u),f') \\
=&\,B_m(t_k(a)t_k'\tilde{w}_k j(a,x,u),f_{\tilde{w}_k}).
\end{align*}

Altogether, we  have
\begin{align}\label{l-2 twist eqn}
\begin{split}
0 =&\,\Psi(W_m^f,\tilde{\xi}_{1-s}^{i,v})-\Psi(W_m^{f'},\tilde{\xi}_{1-s}^{i,v}) \\
=&\,\int_{U_{\GL_k}\backslash\GL_k}\int_{V_k}\left(\int_{r_x\in R^{l,k}} B_m(t_k(a)t_k'\tilde{w}_k j(a,x,u),f_{\tilde{w}_k})dr\right) \\
&\,\cdot \tilde{\xi}_{1-s}^{i,v}(l_k(a)w_ku, I_{l})
|\mathrm{det}(a)|^{-k}duda \\
=&\,\int_{U_{\GL_k}\backslash\GL_k}\int_{V_k}\left(\int_{r_x\in R^{l,k}} B_m(t_k(a)t_k'\tilde{w}_k j(I_k,x,u),f_{\tilde{w}_k})dr\right) \\
&\,\cdot \tilde{\xi}_{1-s}^{i,v}(l_k(a)w_ku, I_{l})
|\mathrm{det}(a)|^{-l}duda,
\end{split}
\end{align}
where we made the change of variables $x\mapsto xa\inv$ in the last step.

Next, let $D_m=\{j(I_k,x,u)\, | \, j(I_k,x,u)\in H_m\}.$ By Lemma \ref{6.3}(3), we have, for $j(I_k,x,u)\notin D_m,$
$$
B_m(t_k(a)t_k'\tilde{w}_k j(I_k,x,u),f_{\tilde{w}_k})=0.
$$
However, by Proposition \ref{partialBesselprop}, for $j(I_k,x,u)\in D_m,$ 
$$
B_m(t_k(a)t_k'\tilde{w}_k j(I_k,x,u),f_{\tilde{w}_k})=B_m(t_k(a)t_k'\tilde{w}_k,f_{\tilde{w}_k}).
$$
Furthermore, for $j(I_k,x,u)\in D_m$, we have $u\in V_{k,m}$ and thus, by Lemma \ref{5.4}, $\tilde{\xi}_{1-s}^{i,v}(l_k(a) w_n u)=\mathrm{Vol}(\ol{V}_{k,i})|\mathrm{det}(a)|^{\frac{1}{2}-s} W_v^*(a).$

From Equation \eqref{l-2 twist eqn}, we obtain
$$
0=\int_{U_{\GL_k}\backslash\GL_k}B_m(t_k(a)t_k'\tilde{w}_k,f_{\tilde{w}_k}) |\mathrm{det}(a)|^{\frac{1}{2}-s-l}W_v^*(a)da.
$$
This holds for any irreducible $\psi\inv$-generic representations $(\tau, V_\tau)$ of $\GL_k$ and $v\in V_\tau.$ From Proposition \ref{JS Prop}, we have
$$
B_m(t_k(a)t_k'\tilde{w}_k,f_{\tilde{w}_k})=0,
$$
for any $a\in\GL_k.$
\end{proof}

As a corollary of the computations above, we obtain the following result on the equality of twisted gamma factors of $\pi$ and $c \cdot \pi$, which can also be implied by the work of \cite{Art13, JL14}. 

\begin{cor}\label{conj gamma1}
Let $\pi$ be an irreducible  $\psi$-generic supercuspidal representation of $\SO_{2l}$. Then $\gamma(s, \pi\times\tau,\psi)=\gamma(s, c\cdot\pi\times\tau,\psi)$ for all irreducible $\psi^{-1}$-generic representations $\tau$ of $\GL_{k}$ for $k\leq l-2$.
\end{cor}

\begin{proof}
Since $c\tilde{t}\inv t_k(a)t_k'\tilde{w}_k \tilde{t}c=t_k(a)t_k'\tilde{w}_k$ for any $a\in\GL_k$, repeating the above arguments for $\pi'=c\cdot\pi$ without the assumption that the $\gamma$-factors are equal gives the claim.
\end{proof}

We are ready to prove Theorem \ref{l-2 theorem}.
\begin{proof}[Proof of Theorem \ref{l-2 theorem}]

Let $w_{\mathrm{max}}=w_l^{L_{k,1^{l-k}}}\tilde{w}_k.$ Then $w_{long}w_{\mathrm{max}}$ is the long Weyl element of $M_{w_{\mathrm{max}}}\cong\GL_{1}^k\times\SO_{2(l-k)}.$ Let $Z(\SO_{2l})$ denote the center of $\SO_{2l}.$ We have 
$$
T_{w_{\mathrm{max}}}=\left\{t_k(a_1,\dots,a_k)t_k' \, | \, a_i\in F^\times \right\} \times Z(\SO_{2l}).
$$
For $t\in T_{w_{\mathrm{max}}},$ write $t=zt_k(a)t_k'$ where $z\in Z(\SO_{2l})$ and $a\in\GL_1^k.$ We have $t_k(a)t_k'w_l^{L_{k,1^{l-k}}}=t_k(a)w_l^{L_{k,1^{l-k}}}t_k'=t_n(b)t_k'$ for some $b\in\GL_k.$ Thus, for any $t\in T_{w_{\mathrm{max}}},$
$$
B_m(tw_{\mathrm{max}},f_{\tilde{w}_k})=\omega(z)B_m(t_k(a)w_l^{L_{k,1^{l-k}}}t_k'\tilde{w}_k,f_{\tilde{w}_k})=0,
$$
by Proposition \ref{l-2 prop}.

Next, suppose that $w\in \mathrm{B}(\SO_{2l})$ with $\tilde{w}_k\leq w \leq w_{\mathrm{max}}.$ By Lemma \ref{CPSSh}, $A_w\subseteq A_{w_{\mathrm{max}}}.$ Also $[\tilde{w}_k, w_{\mathrm{max}}]\cap\mathrm{B}(\SO_{2l})=\mathrm{B}_k(\SO_{2l})$. Again, from Proposition \ref{l-2 prop}, we find
$$
B_m(tw,f_{\tilde{w}_k})=0,
$$
for any $t\in T_w.$

Let 
$$
\Omega_{\tilde{w}_k}'=\bigcup_{\substack{w\in \mathrm{B}(\SO_{2l}), w>w_{\mathrm{max}}, \\ d_{B}(w,w_{\mathrm{max}})=1}} \Omega_w.
$$
By Theorem \ref{CST}, there exists $f_{\tilde{w}_k}'\in C_c^\infty(\Omega_{\tilde{w}_k}',\omega)$ such that
$$
B_m(g, f_{\tilde{w}_k})=B_m(g,f_{\tilde{w}_k}'),
$$
for any $g\in\SO_{2l}$ and $m$ large enough (depending only $f_{\tilde{w}_k}$ and hence $f$ and $f'$).

We describe the set $\{w\in \mathrm{B}(\SO_{2l}) \, | \, w>w_{\mathrm{max}},  d_{B}(w,w_{\mathrm{max}})=1\}.$ From the proof of Lemma \ref{6.3}(2), we have $\theta_{w_{\mathrm{max}}}=\{\alpha_{k+1},\dots,\alpha_l\}.$ Since $w>w_{\mathrm{max}}$ and  $d_{B}(w,w_{\mathrm{max}})=1,$ we must have $\theta_w=\theta_{w_{\mathrm{max}}}\backslash\{\alpha_i\}$ for some $k+1\leq i \leq l.$ Set $w_i'=w_{\theta_{w_{\mathrm{max}}}\backslash\{\alpha_i\}}.$ Then,
$$
\{w\in \mathrm{B}(\SO_{2l}) \, | \, w>w_{\mathrm{max}},  d_{B}(w,w_{\mathrm{max}})=1\}=\{w_i' \, | \, k+1\leq i \leq l \}.
$$
Thus, 
$$\Omega_{\tilde{w}_k}'=\bigcup_{i=k+1}^l \Omega_{w_i'}.$$ By a partition of unity argument (similar to the proof of Corollary \ref{C(0)}), there exists $f_{w_i'}\in C_c^\infty(\Omega_{w_i'},\omega)$ for $k+1\leq i \leq l$ such that
$$
f_{\tilde{w}_k}'=\sum_{i=k+1}^l f_{w_i'}.
$$
So,
\begin{equation}\label{inductive step eqn}
B_m(g,f_{\tilde{w}_k})=B_m(g,f_{\tilde{w}_k}')=\sum_{i=k+1}^l B_m(g,f_{w_i'}),
\end{equation}
for any $g\in\SO_{2l}$ and $m$ large enough.

Finally, fix $i$ such that $k+1\leq i \leq l.$ We have $\theta_{\tilde{w}_i}=\Delta(\SO_{2l})\backslash\{\alpha_i\}$. Since $\theta_{w_i'}\subseteq\theta_{w_{\mathrm{max}}}\backslash\{\alpha_i\}\subseteq\theta_{\tilde{w}_i},$ it follows that $w_i'>\tilde{w}_i$ and hence $\Omega_{w_i'}\subseteq\Omega_{\tilde{w}_i}.$ By Proposition \ref{Omega_w}, $\Omega_{w_i'}$ is open in $\Omega_{\tilde{w}_i}.$ Hence, $C_c^\infty(\Omega_{w_i'},\omega)\subseteq C_c^\infty(\Omega_{\tilde{w}_i},\omega).$ Thus $f_{w_i}'\in C_c^\infty(\Omega_{w_i'},\omega)$. Set
$$
f_{\tilde{w}_i}'= f_{w_i}' + f_{\tilde{w}_i}\in C_c^\infty(\Omega_{w_i'},\omega).
$$
Then, by the inductive hypothesis and Equation \eqref{inductive step eqn}, we have
$$
B_m(g,f)-B_m(g,f')=\sum_{i=k+1}^l B_m(g,f_{\tilde{w}_i}').
$$
for any $g\in\SO_{2l}$ and $m$ large enough depending only on $f$ and $f'.$
This completes the proof of Theorem \ref{l-2 theorem}.
\end{proof}

\subsection{Twists up to $\GL_{l-1}$}

From Theorem \ref{l-2 theorem}, we have that $\mathcal{C}(l-2)$ implies
\begin{equation}\label{Bessel difference at l-2}
    B_m(g,f)-B_m(g,f')= B_m(g,f_{\tilde{w}_l}) + B_m(g,f_{\tilde{w}_l^c}),
\end{equation}
for any $g\in\SO_{2l}$ and $m$ large enough. However, from Lemma \ref{obstruction}, $B_m(g,f_{\tilde{w}_l})$ and $B_m(g,f_{\tilde{w}_l^c})$ have an overlap in their support. Consequently, the condition $\mathcal{C}(l-1)$ is unable to remove one of them. Instead, we follow roughly the same argument as given for Theorem \ref{l-2 theorem}. The zeta integrals attached to $\mathcal{C}(l-1)$ are reduced to partial Bessel functions with arguments $t_{l-1}(a)t_{l-1}'\tilde{w}_{l-1}$ where $a\in\GL_{l-1}.$ This is because conjugation by $\tilde{t}c$ fixes $t_{l-1}(a)t_{l-1}'\tilde{w}_{l-1}$. The rest of the Bruhat cell is determined by the condition $\mathcal{C}(l)$ as in the case of finite fields \cite{HL22}. First, we address $\mathcal{C}(l-1).$

\begin{thm}\label{l-1 thm}
The condition $\mathcal{C}(l-1)$ implies
$$
B_m(t_{l-1}(a)t_{l-1}'\tilde{w}_{l-1},f_{\tilde{w}_l})+B_m(t_{l-1}(a)t_{l-1}'\tilde{w}_{l-1},f_{\tilde{w}_l^c})=0,
$$
for any $a\in\GL_{l-1}$ and $m$ large enough depending only on $f$ and $f'.$
\end{thm}

\begin{proof}
Let $m$ be large enough such that Theorem \ref{l-2 theorem} and Lemma \ref{6.3} hold. Let $(\tau, V_\tau)$ be an irreducible $\psi\inv$-generic representation of $\GL_{l-1}$ and $v\in V_\tau.$ We consider the embedding of $\SO_{2l-1}$ into $\SO_{2l}$ as in $\S$\ref{gps and reps}. Let $j:\SO_{2l-1}\rightarrow\SO_{2l}$ denote the embedding. Recall that $Q_{l-1}=L_{l-1} V_{l-1}$ is the standard Siegel parabolic subgroup of $\SO_{2l-1}$ with Levi subgroup $L_{l-1}\cong\GL_{l-1}.$ Note that $w^{l,l-1}$ is trivial. Let $V_{{l-1},m}=\{u\in V_{l-1} \, |  j(u) \in H_m\}.$ $V_{{l-1},m}$ is an open compact subgroup of $V_{l-1}.$ Let $i$ be a positive integer with $i\geq\mathrm{max}\{m,i_0(v),i(V_{{l-1},m},v)\}$ (for notation, see Lemma \ref{5.3} for $i_0(v)$ and Lemma \ref{5.4} for $i(V_{{l-1},m},v)$ where we take $D=V_{{l-1},m}$). We fix the section $\xi_s^{i,v}\in I(\tau,s,\psi\inv)$ defined in $\S\ref{sections}.$ We also consider the section $\tilde{\xi}_{1-s}^{i,v}=M(\tau,s,\psi\inv)\xi_s^{i,v}\in I(\tau^*,1-s,\psi\inv)$. 

First, we compute the non-intertwined zeta integrals $\Psi(W_m^f,\xi_s^{i,v}).$ We have that $V_{l-1} L_{l-1} \ol{V}_{l-1}$ is dense in $\SO_{2l-1}$ where $\ol{V}_{l-1}$ denotes the opposite unipotent radical to $V_{l-1}.$ Hence, we can replace the integral over $U_{\SO_{2l-1}}\backslash\SO_{2l-1}$ with an integral over $U_{\SO_{2l-1}}\backslash V_{l-1} L_{l-1} \ol{V}_{l-1}\cong U_{\GL_{l-1}}\backslash\GL_{l-1} \times \ol{V}_{l-1}$ where $\GL_{l-1}$ is realized in $L_{l-1}\subseteq\SO_{2l-1}.$ For $g=u_1 l_{l-1}(a) \ol{u}_2\in U_{\GL_k}\backslash\GL_{l-1} \times \ol{V}_{l-1}$ we use the quotient Haar measure $dg=|\mathrm{det}(a)|^{-({l-1})} d\ol{u}_2 da.$ 

Note that $R^{l,l-1}$ is trivial. We have
\begin{align*}
    \Psi(W_m^f,\xi_s^{i,v})=\int_{U_{\GL_{l-1}}\backslash\GL_{l-1}}\int_{\ol{V}_{l-1}} W_m^f(  j(l_{l-1}(a)\ol{u})) \xi_s^{i,v}(l_{l-1}(a)\ol{u}, I_{l})|\mathrm{det}(a)|^{-({l-1})}d\ol{u}da.
\end{align*}
By Equation \eqref{xi nonintertwined}, we have
$$
\Psi(W_m^f,\xi_s^{i,v})=\int_{U_{\GL_{l-1}}\backslash\GL_{l-1}}\int_{\ol{V}_{{l-1},i}}W_m^f( j(l_{l-1}(a)\ol{u})) |\mathrm{det}(a)|^{s-\frac{1}{2}-({l-1})} W_v(a)d\ol{u}da.
$$
Furthermore, for $\ol{u}\in\ol{V}_{{l-1},i}$, we have $j(\ol{u}) \in H_m.$ By Proposition \ref{partialBesselprop}, for any $g\in\SO_{2l},$
$$
W^f_m(g j(\ol{u}))= W_m^f(g).
$$
Thus,
$$
\Psi(W_m^f,\xi_s^{i,v})=\mathrm{Vol}(\ol{V}_{{l-1},i})\int_{U_{\GL_{l-1}}\backslash\GL_{l-1}}W_m^f( j(l_{l-1}(a))  |\mathrm{det}(a)|^{s-\frac{1}{2}-({l-1})} W_v(a)da.
$$
Let $q_{l-1}(a)=\mathrm{diag}(a,I_2,a^*).$ We have $j(l_{l-1}(a))=q_{l-1}(a).$ 
Thus, by Lemma \ref{6.3}(1),
$$
W_m^f( j(l_{l-1}(a)) )=W_m^{f'}( j(l_{l-1}(a)) ),
$$
and hence 
$$
\Psi(W_m^f,\xi_s^{i,v})=\Psi(W_m^{f'},\xi_s^{i,v}).
$$

By assumption, $\gamma(s,\pi\times\tau,\psi)=\gamma(s,\pi'\times\tau,\psi).$ Since $\Psi(W_m^f,\xi_s^{i,v})=\Psi(W_m^{f'},\xi_s^{i,v}),$ the local functional equation gives
$$\Psi(W_m^f,\tilde{\xi}_{1-s}^{i,v})=\Psi(W_m^{f'},\tilde{\xi}_{1-s}^{i,v}).$$

We have that $V_{l-1} L_{l-1} w_{l-1} V_{l-1}$ is dense in $\SO_{2l-1}.$ Hence, we can replace the integral over $U_{\SO_{2l-1}}\backslash\SO_{2l-1}$ with an integral over $U_{\SO_{2l-1}}\backslash V_{l-1} L_{l-1} w_{l-1} V_{l-1}$ which is isomorphic to $U_{\GL_{l-1}}\backslash\GL_{l-1} w_{l-1} V_{l-1}.$ We obtain
\begin{align*}
   &\,\Psi(W_m^f,\tilde{\xi}_{1-s}^{i,v})\\
   =&\, \int_{U_{\GL_{l-1}}\backslash\GL_{l-1}}\int_{V_{l-1}} W_m^f(  j(l_{l-1}(a)w_{l-1} u))    \tilde{\xi}_{1-s}^{i,v}(l_{l-1}(a)w_{l-1} u, I_{l})|\mathrm{det}(a)|^{-({l-1})}duda. 
\end{align*}
We have a similar expression for $\Psi(W_m^{f'},\tilde{\xi}_{1-s}^{i,v}).$

Note that $j(w_{l-1})=t_{l-1}'\tilde{w}_{l-1}$ and $ j(l_{l-1}(a)))=\mathrm{diag}(a,I_{2},a^*)=t_{l-1}(a).$ For 
$$u=\left(\begin{matrix}
I_{l-1} & y_1   & y \\
    & 1     & y_1' \\
    &       & I_{l-1}
\end{matrix}\right)\in V_{l-1},$$ we have 
$$
 j(u) = \left(\begin{matrix}
I_{l-1} & *  & * \\
&  I_2  & * \\
& &  I_{l-1}
\end{matrix}\right).
$$
Recall $P_{l-1}$ is the standard maximal parabolic subgroup of $\SO_{2l}$ with Levi subgroup $M_{l-1}$ isomorphic to $\GL_{l-1}\times\SO_{2}$ and unipotent radical $N_{l-1}.$
We have
$$
 j(l_{l-1}(a)w_{l-1} u)  
=t_{l-1}(a)t_{l-1}'\tilde{w}_{l-1} j(u)\inv\in P_{l-1} \tilde{w}_{l-1} P_{l-1}.
$$
At this point, we would like to use Lemma \ref{6.3}(2). However, this is not true in our case (see Lemma \ref{obstruction}). Instead, we use Theorem \ref{l-2 theorem} to obtain
\begin{align*}
&\,W_m^f(t_{l-1}(a)t_{l-1}'\tilde{w}_{l-1} j(u))-W_m^{f'}(t_{l-1}(a)t_{l-1}'\tilde{w}_{l-1} j(u))
\\
=&\,B_m(t_{l-1}(a)t_{l-1}'\tilde{w}_{l-1} j(u),f)-B_m(t_{l-1}(a)t_{l-1}'\tilde{w}_{l-1} j(u),f')
\\
=&\,B_m(t_{l-1}(a)t_{l-1}'\tilde{w}_{l-1} j(u),f_{\tilde{w}_{l}})+B_m(t_{l-1}(a)t_{l-1}'\tilde{w}_{l-1} j(u),f_{\tilde{w}_{l}^c}).
\end{align*}
For brevity, let $B(g)=B_m(g,f_{\tilde{w}_{l}})+B_m(g,f_{\tilde{w}_{l}^c})$ for $g\in\SO_{2l}(F).$

Altogether, we  have
\begin{align}\label{l-1 twist eqn}
\begin{split}
0&=\Psi(W_m^f,\tilde{\xi}_{1-s}^{i,v})-\Psi(W_m^{f'},\tilde{\xi}_{1-s}^{i,v}) \\
&=\int_{U_{\GL_{l-1}}\backslash\GL_{l-1}}\int_{V_{l-1}} B(t_{l-1}(a)t_{l-1}'\tilde{w}_{l-1} j(u))   \tilde{\xi}_{1-s}^{i,v}(l_{l-1}(a)w_{l-1} u, I_{l})|\mathrm{det}(a)|^{-({l-1})}duda. 
\end{split}
\end{align}

By Lemma \ref{6.3}(3), for $u\notin V_{l-1,m},$
$$
B(t_{l-1}(a)t_{l-1}'\tilde{w}_{l-1} j(u))=0.
$$
However, by Proposition \ref{partialBesselprop}, for $u\in V_{l-1,m},$ 
$$
B(t_{l-1}(a)t_{l-1}'\tilde{w}_{l-1} j(u))=B(t_{l-1}(a)t_{l-1}'\tilde{w}_{l-1}).
$$
By Equation \eqref{xi intertwined}, for $u\in V_{l-1,m}$, $$\tilde{\xi}_{1-s}^{i,v}(l_{l-1}(a) w_{l-1} u, I_l)=\mathrm{Vol}(\ol{V}_{{l-1},i})|\mathrm{det}(a)|^{\frac{1}{2}-s} W_v^*(a).$$

From Equation \eqref{l-1 twist eqn}, we obtain
$$
0=\int_{U_{\GL_{l-1}}\backslash\GL_{l-1}}B(t_{l-1}(a)t_{l-1}'\tilde{w}_{l-1}) |\mathrm{det}(a)|^{\frac{1}{2}-s-(l-1)}W_v^*(a)da.
$$
This holds for any irreducible $\psi\inv$-generic representations $(\tau, V_\tau)$ of $\GL_{l-1}$ and $v\in V_\tau.$ By Proposition \ref{JS Prop}, we have
$$
B(t_{l-1}(a)t_{l-1}'\tilde{w}_{l-1},f_{\tilde{w}_{l-1}})=B_m(t_{l-1}(a)t_{l-1}'\tilde{w}_{l-1},f_{\tilde{w}_{l}})+B_m(t_{l-1}(a)t_{l-1}'\tilde{w}_{l-1},f_{\tilde{w}_{l}^c})=0,
$$
for any $a\in\GL_{l-1}.$
This completes the proof of Theorem \ref{l-1 thm}.
\end{proof}

Similarly, we have the following corollary. 

\begin{cor}\label{conj gamma2}
Let $\pi$ be an irreducible  $\psi$-generic supercuspidal representation of $\SO_{2l}$. Then $\gamma(s, \pi\times\tau,\psi)=\gamma(s, c\cdot\pi\times\tau,\psi)$ for all irreducible $\psi^{-1}$-generic representations $\tau$ of $\GL_{l-1}.$
\end{cor}

\begin{proof}
By Corollary \ref{conj gamma1}, $\gamma(s, \pi\times\tau,\psi)=\gamma(s, c\cdot\pi\times\tau,\psi)$ for all irreducible $\psi^{-1}$-generic representations $\tau$ of $\GL_{k}$ for $k\leq l-2.$ The rest of the proof is similar to that of Corollary \ref{conj gamma1}.
\end{proof}

\subsection{Twists up to $\GL_l$}

We consider the twists up to $\GL_l$.
Recall from Equation \eqref{Bessel difference at l-2} that
$$
B_m(g,f)-B_m(g,f')= B_m(g,f_{\tilde{w}_l}) + B_m(g,f_{\tilde{w}_l^c}).
$$
Thus, we have a similar result for the conjugation of the partial Bessel functions 
$$
B_m^c(g,f)-B_m^c(g,f')= B_m^c(g,f_{\tilde{w}_l}) + B_m^c(g,f_{\tilde{w}_l^c}).
$$
The next theorem shows that $B_m(g,f)-B_m(g,f')+B_m^c(g,f)-B_m^c(g,f')=0$ for $m$ large enough enough.

\begin{thm}\label{l thm}
The condition $\mathcal{C}(l)$ implies
$$
B_m(g,f_{\tilde{w}_l})+B_m(g,f_{\tilde{w}_l^c})+B_m^c(g,f_{\tilde{w}_l})+B_m^c(g,f_{\tilde{w}_l^c})=0,
$$
for any $g\in\SO_{2l}$ and $m$ large enough depending only on $f$ and $f'.$
\end{thm}

\begin{proof}
Let $m$ be large enough such that Theorem \ref{l-1 thm} and Lemma \ref{6.3} hold. Let $(\tau, V_\tau)$ be an irreducible $\psi\inv$-generic representation of $\GL_{l}$ and $v\in V_\tau.$ We consider the embedding of $\SO_{2l}$ into $\SO_{2l+1}$ as in $\S$\ref{gps and reps}. Let $j:\SO_{2l}\rightarrow\SO_{2l+1}$ denote the embedding. Recall that $Q_{l}=L_{l} V_{l}$ is the standard Siegel parabolic subgroup of $\SO_{2l+1}$ with Levi subgroup $L_{l}\cong\GL_{l}.$ Also, $P_{l}=M_{l} N_{l}$ is the non-conjugated Siegel parabolic subgroup of $\SO_{2l}$ with Levi subgroup $M_{l}\cong\GL_{l}.$ Let $\ol{V}_{{l},m}=\ol{V}_{l} \cap H_m^{\SO_{2l+1}}$ and $\ol{N}_{l,m}=\{u\in \ol{N}_{l} \, |  j(u) \in \ol{V}_{l,m}\}.$ $j(\ol{N}_{{l},m})$ is an open compact subgroup of $\ol{V}_{l}.$ Let $i$ be a positive integer with $i\geq\mathrm{max}\{m,i_0(v),i(j(N_{{l},m}),v)\}$ (for notation, see Lemma \ref{5.3} for $i_0(v)$ and Lemma \ref{5.4} for $i(j(N_{{l},m}),v)$ where we take $D=j(N_{{l},m})$). We consider the sections $\xi_s^{i,v}\in I(\tau,s,\psi\inv)$ and $\tilde{\xi}_{1-s}^{i,v}=M(\tau,s,\psi\inv)\xi_s^{i,v}\in I(\tau^*,1-s,\psi\inv)$ defined in $\S\ref{sections}.$

First, we compute the non-intertwined zeta integrals $\Psi(W_m^f,\xi_s^{i,v}).$ We have that $N_{l} M_{l} \ol{N}_{l}$ is dense in $\SO_{2l}$ where $\ol{N}_{l}$ denotes the opposite unipotent radical to $N_{l}.$ Hence, we can replace the integral over $U_{\SO_{2l}}\backslash\SO_{2l}$ with an integral over $U_{\SO_{2l}}\backslash N_{l} M_{l} \ol{N}_{l}\cong U_{\GL_{l}}\backslash\GL_{l} \times \ol{N}_{l}$ where $\GL_{l}$ is realized in $M_{l}\subseteq\SO_{2l}.$ Note for $g=u_1 m_{l}(a) \ol{u}_2\in U_{\GL_l}\backslash\GL_{l} \times \ol{N}_{l}$ we use the quotient Haar measure $dg=|\mathrm{det}(a)|^{-{l}} d\ol{u}_2 da.$

We have
$$
\Psi(W_m^f,\xi_s^{i,v})=\int_{U_{\GL_{l}}\backslash\GL_{l}}\int_{\ol{N}_{l}} W_m^f(  l_{l}(a)\ol{u}) \xi_s^{i,v}(w_{l,l}j(l_{l}(a)\ol{u}), I_{l})|\mathrm{det}(a)|^{-{l}}d\ol{u}da.
$$
By Equation \eqref{xi nonintertwined}, we have
$$
\Psi(W_m^f,\xi_s^{i,v})=\int_{U_{\GL_{l}}\backslash\GL_{l}}\int_{\ol{N}_{l,i}} W_m^f(  l_{l}(a)\ol{u}) \xi_s^{i,v}(w_{l,l}j(l_{l}(a)\ol{u}), I_{l})|\mathrm{det}(a)|^{-{l}}d\ol{u}da.
$$
Furthermore, for $\ol{u}\in\ol{N}_{{l},i}$, we have $\ol{u} \in H_m$. By Proposition \ref{partialBesselprop},
$$
W^f_m(g j(\ol{u}))= W_m^f(g),
$$
for any $g\in\SO_{2l}(F).$ From this and Equation \eqref{xi nonintertwined}, we have
$$
\Psi(W_m^f,\xi_s^{i,v})=\mathrm{Vol}(\ol{N}_{l,i})\int_{U_{\GL_{l}}\backslash\GL_{l}} W_m^f(  l_{l}(a)) \xi_s^{i,v}(w_{l,l}j(l_{l}(a)), I_{l})|\mathrm{det}(a)|^{-{l}}da.
$$
By Lemma \ref{6.3}(1),
$$
W_m^f( m_l(a)) )=W_m^{f'}( m_l(a)),
$$
and hence 
$$
\Psi(W_m^f,\xi_s^{i,v})=\Psi(W_m^{f'},\xi_s^{i,v}).
$$

By assumption, $\gamma(s,\pi\times\tau,\psi)=\gamma(s,\pi'\times\tau,\psi).$ Since $\Psi(W_m^f,\xi_s^{i,v})=\Psi(W_m^{f'},\xi_s^{i,v}),$ the local functional equation gives
$$\Psi(W_m^f,\tilde{\xi}_{1-s}^{i,v})=\Psi(W_m^{f'},\tilde{\xi}_{1-s}^{i,v}).$$
By Theorem \ref{l-2 theorem}, $$
B_m(g,f)-B_m(g,f')= B_m(g,f_{\tilde{w}_l}) + B_m(g,f_{\tilde{w}_l^c}),
$$
for any $g\in\SO_{2l}(F).$ Let $B_m(g):=B_m(g,f_{\tilde{w}_l}) + B_m(g,f_{\tilde{w}_l^c}).$ $B_m$ is supported on $\Omega_{\tilde{w}_l}\cup\Omega_{\tilde{w}_l^c}$ and hence
\begin{align}\label{l twist integrals}
\begin{split}
0&=\Psi(W_m^f,\tilde{\xi}_{1-s}^{i,v})-\Psi(W_m^{f'},\tilde{\xi}_{1-s}^{i,v})\\
&=\int_{U_{\SO_{2l}}\backslash \Omega_{\tilde{w}_l}\cup\Omega_{\tilde{w}_l^c}} B_m(g) \tilde{\xi}_{1-s}^{i,v}(w_{l,l}j(g), I_{l})dg \\
&=\int_{U_{\SO_{2l}}\backslash (\Omega_{\tilde{w}_l}\backslash\Omega_{\tilde{w}_l^c})} B_m(g) \tilde{\xi}_{1-s}^{i,v}(w_{l,l}j(g), I_{l})dg \\ &\quad+\int_{U_{\SO_{2l}}\backslash(\Omega_{\tilde{w}_l^c}\backslash\Omega_{\tilde{w}_l})} B_m(g) \tilde{\xi}_{1-s}^{i,v}(w_{l,l}j(g), I_{l})dg \\
&\quad
+\int_{U_{\SO_{2l}}\backslash (\Omega_{\tilde{w}_l}\cap\Omega_{\tilde{w}_l^c})} B_m(g) \tilde{\xi}_{1-s}^{i,v}(w_{l,l}j(g), I_{l})dg.
\end{split}
\end{align}

We  aim to combine the second and third integrals. Let us outline the next several steps. First, we conjugate the second integral by $c\tilde{t}\inv.$ After a change of variables, we find that 
\begin{align}
\begin{split}
&\,\int_{U_{\SO_{2l}}\backslash (\Omega_{\tilde{w}_l}\backslash\Omega_{\tilde{w}_l^c})} B_m(g) \tilde{\xi}_{1-s}^{i,v}(w_{l,l}j(g), I_{l})dg \\ +&\,\int_{U_{\SO_{2l}}\backslash(\Omega_{\tilde{w}_l^c}\backslash\Omega_{\tilde{w}_l})} B_m(g) \tilde{\xi}_{1-s}^{i,v}(w_{l,l}j(g), I_{l})dg \\
=&\,\int_{U_{\SO_{2l}}\backslash (\Omega_{\tilde{w}_l}\backslash\Omega_{\tilde{w}_l^c})} (B_m+B_m^c)(g) \tilde{\xi}_{1-s}^{i,v}(w_{l,l}j(g), I_{l})dg. 
\end{split}
\end{align}
$B_m^c$ is the difference of the partial Bessel functions for the conjugate representations.

The last integral in Equation \eqref{l twist integrals} must be dealt with delicately as it folds to itself. However, we find that the intersection embeds into the support of $\tilde{\xi}_{1-s}^{i,v}$ only when we are outside the set determined by Theorem \ref{l-1 thm}. The rest of this set embeds into the support; however, the argument of $\tilde{\xi}_{1-s}^{i,v}$ (when reduced) is a two-to-one function. This would cause issues when applying Proposition \ref{JS Prop}. To circumvent this, we decompose the integral into two integrals and recombine them such that the argument can be mapped one-to-one with the integrand. We can then apply Proposition \ref{JS Prop}. The main difference between this and the finite field case in \cite{HL22} is that we need to handle a unipotent element (due to the local case using only partial Bessel functions). We see later that it lies in an open compact subgroup and then factors out via a volume term. However, this simplifies the calculation.

Let us begin the details. First, consider the integral
$$
\int_{U_{\SO_{2l}}\backslash(\Omega_{\tilde{w}_l^c}\backslash\Omega_{\tilde{w}_l})} B_m(g) \tilde{\xi}_{1-s}^{i,v}(w_{l,l}j(g), I_{l})dg.
$$
By Equation \eqref{xi intertwined}, the support of $\tilde{\xi}_{1-s}^{i,v}(w_{l,l}j(g), I_{l})$ is contained in $Q_lw_l V_l$ where $Q_l=L_lV_l$ is the standard Siegel parabolic of $\SO_{2l+1}.$ From Section \S\ref{Bruhat order}, we have that 
$$
\{w\in \mathrm{B}(\SO_{2l}) \, | \, C(w)\subseteq \Omega_{\tilde{w}_l^c}\backslash\Omega_{\tilde{w}_l} \}=B_l^c(\SO_{2l}),
$$
and
$$
\{w\in \mathrm{B}(\SO_{2l}) \, | \, C(w)\subseteq \Omega_{\tilde{w}_l}\backslash\Omega_{\tilde{w}_l^c} \}=B_l(\SO_{2l}).
$$
By \cite[Propositions 8.3 and 8.4]{HL22}, we see that for any $g\in\Omega_{\tilde{w}_l^c}\backslash\Omega_{\tilde{w}_l}\cup\Omega_{\tilde{w}_l}\backslash\Omega_{\tilde{w}_l^c}$, $j(g)\in Q_l w_l V_l.$ Let $g=tt_l^c(w)\tilde{w}_l^c u$ where $t\in T_{\SO_{2l}}$ $t_l^c(w)\tilde{w}_l^c\in B_l^c(\SO_{2l})$ and $u\in U_{\SO_{2l}}$. Then, $j(g)=j(tt_l^c(w))j(\tilde{w}_l^c)j(u)$. Recall that we set $D=j(N_{l,m})$. Thus, from Equation \eqref{xi intertwined}, we may restrict to $U_{\SO_{2l}}$ to the subgroup $N_{l,m}\subseteq N_l\cap U_m.$

We relate Bruhat cells for $\mathrm{B}_l^c(\SO_{2l})$ with Bruhat cells for $\mathrm{B}_l(\SO_{2l})$ as follows. First, by Equation \eqref{Besseleqn}, $B_m(twu)=B_m(tw)\psi(u)$ for any $t\in T_{\SO_{2l}},$ $u\in N_{l,m},$ and $w\in W(\SO_{2l}).$ We have
$$B_m(tw)
=B_m(c\tilde{t}\inv \tilde{t} ctw c\tilde{t}\inv \tilde{t} c)
=B_m^c( \tilde{t} ctw c\tilde{t}\inv)
=B_m^c( t' cwc),
$$
where $t'=\tilde{t}ctc (cwc\tilde{t}\inv (cwc)\inv).$ For $w\in\mathrm{B}_l^c(\SO_{2l})$, we have $cwc\in \mathrm{B}_l(\SO_{2l}).$ By \cite[Proposition 4.5]{HL22}, there exists $w^\prime\in W(\GL_l)$ such that $cwc=t_l(w^\prime)\tilde{w}_l$ where $w'=(w_{i,j}')_{i,j=1}^l\in W(\GL_l)$ with $w_{l,1}'=0.$ We also let $(w')^*=(w_{i,j}'^*)_{i,j=1}^l$ and $r$ be such that $w_{r,1}'=1.$ 

For $l$ even, $(cwc\tilde{t}\inv (cwc)\inv)=t_l(\mathrm{diag}(1,\dots, 1, -2, 1,\dots, 1))$ where $-2$ is the $r$-th entry. Thus, $t'=t_l(\mathrm{diag}(t_1,\dots, t_{r-1}, -2t_r, t_{r+1},\dots, t_{l-1}, \frac{-t_l\inv}{2})).$ Let $a_1,a_2\in\GL_l$ and $n_1,n_2,n_3,n_4\in V_l$ be such that 
$tw=l_l(a_1)n_1 w_l n_2$ and $tcwc=l_l(a_2)n_3 w_l n_4.$ By the proofs of \cite[Propositions 8.3 and 8.4]{HL22}, we have
$$
a_1^*=\left(\begin{matrix}
\frac{-t_l(w_{1,j}'^*)_{j=1}^{l-1}}{2} & \frac{1}{4} \\
\mathrm{diag}(t_{l-1}\inv,\dots, t_1\inv) (w_{i,j}'^*)_{i=2,j=1}^{i=l,j=l-1} & \mathrm{diag}(t_{l-1}\inv,\dots, t_1\inv)\frac{-(w_{i,l}'^*)_{i=2}^l}{2} \\
\end{matrix}\right),$$ and
$$
a_2^*=\left(\begin{matrix}
\frac{t_l\inv(w_{1,j}'^*)_{j=1}^{l-1}}{4} & \frac{1}{4} \\
\mathrm{diag}(t_{l-1}\inv,\dots, t_1\inv) (w_{i,j}'^*)_{i=2,j=1}^{i=l,j=l-1} & \mathrm{diag}(t_{l-1}\inv,\dots, t_1\inv)\frac{(w_{i,l}'^*)_{i=2}^l}{4} \\
\end{matrix}\right).
$$
Performing the change of variables $t_r\mapsto \frac{-t_r}{2}$ and $t_l\mapsto \frac{-t_l\inv}{2}$ takes $t'\mapsto t$ and $a_1^*\mapsto a_2^*.$ Note that $w'^*_{l+1-r,l}=1$ and this is the coordinate for $t_r\inv.$ Thus we have
\begin{align*}
&\,\sum_{w\in \mathrm{B}_l^c(\SO_{2l})}\int_{ T_{\SO_{2l}}\times U_{\SO_{2l}}}
B_m(twu)\tilde{\xi}_{1-s}^{i,v}(w_{l,l} j(twu), I_l)dtdu \\
=
&\,\sum_{w\in \mathrm{B}_l^c(\SO_{2l})}\int_{ T_{\SO_{2l}}\times U_{\SO_{2l}}}
B_m^c(t'cwcu)\tilde{\xi}_{1-s}^{i,v}(w_{l,l} j(twu), I_l)dtdu \\
=
&\,\sum_{w\in \mathrm{B}_l(\SO_{2l})}\int_{ T_{\SO_{2l}}\times U_{\SO_{2l}}}
B_m^c(twu)\tilde{\xi}_{1-s}^{i,v}(w_{l,l} j(twu), I_l)dtdu.
\end{align*}
Since 
\begin{align*}
    &\,\int_{U_{\SO_{2l}}\backslash(\Omega_{\tilde{w}_l^c}\backslash\Omega_{\tilde{w}_l})} B_m(g) \tilde{\xi}_{1-s}^{i,v}(w_{l,l}j(g), I_{l})dg \\
    =&\,\sum_{w\in \mathrm{B}_l^c(\SO_{2l})}\int_{ T_{\SO_{2l}}\times U_{\SO_{2l}}}
B_m(twu)\tilde{\xi}_{1-s}^{i,v}(w_{l,l} j(twu), I_l)dtdu,
\end{align*}
and
\begin{align*}
&\,\int_{U_{\SO_{2l}}\backslash(\Omega_{\tilde{w}_l}\backslash\Omega_{\tilde{w}_l^c})} B_m(g) \tilde{\xi}_{1-s}^{i,v}(w_{l,l}j(g), I_{l})dg \\
=&\,\sum_{w\in \mathrm{B}_l(\SO_{2l})}\int_{ T_{\SO_{2l}}\times U_{\SO_{2l}}}
B_m(twu)\tilde{\xi}_{1-s}^{i,v}(w_{l,l} j(twu), I_l)dtdu,
\end{align*}
it follows that 
\begin{align*}
&\,\int_{U_{\SO_{2l}}\backslash(\Omega_{\tilde{w}_l}\backslash\Omega_{\tilde{w}_l^c})} B_m(g) \tilde{\xi}_{1-s}^{i,v}(w_{l,l}j(g), I_{l})dg \\ +&\,\int_{U_{\SO_{2l}}\backslash(\Omega_{\tilde{w}_l^c}\backslash\Omega_{\tilde{w}_l})} B_m(g) \tilde{\xi}_{1-s}^{i,v}(w_{l,l}j(g), I_{l})dg \\
=&\,\int_{U_{\SO_{2l}}\backslash(\Omega_{\tilde{w}_l}\backslash\Omega_{\tilde{w}_l^c})} (B_m+B_m^c)(g) \tilde{\xi}_{1-s}^{i,v}(w_{l,l}j(g), I_{l})dg.
\end{align*}
The case for $l$ odd is similar. We omit the details.

Next, we must handle the integral over the intersection. That is, we consider
$$
\int_{U_{\SO_{2l}}\backslash (\Omega_{\tilde{w}_l}\cap\Omega_{\tilde{w}_l^c})} B_m(g) \tilde{\xi}_{1-s}^{i,v}(w_{l,l}j(g), I_{l})dg.
$$
We have $U_{\SO_{2l}}\backslash (\Omega_{\tilde{w}_l}\cap\Omega_{\tilde{w}_l^c})=\sqcup_{w\in\mathrm{B}_{l-1}(\SO_{2l})} T_{\SO_{2l}} w U_{\SO_{2l}}.$ If we try to repeat the above arguments in this case, we run into two issues. First, by \cite[Proposition 8.2]{HL22}, $twu\in Q_l w_l V_l$ if and only if $t=\mathrm{diag}(t_1,\dots,t_l,t_l\inv,\dots,t_1\inv)$ where $t_l\neq 1$ if $l$ is odd and $t_l\neq \frac{-1}{2}$ if $l$ is even. Thus, we must restrict our torus to the subset
$$T_l=\{t=\mathrm{diag}(t_1,\dots,t_l,t_l\inv, \dots, t_1\inv)\in T_{\SO_{2l}} \, | \, t_l\neq \pm c_l\},$$
where $c_l= 1$ if $l$ is odd or $\frac{1}{2}$ if $l$ is even. Second, $cwc=w$ for any $w\in \mathrm{B}_{l-1}(\SO_{2l})$ and this would send
$B_m$ to $B_m^c$; however, to apply Proposition \ref{JS Prop}, we need to obtain $B_m+B_m^c$. To remedy this, we partition $T_l$ into two halves and then recombine them.

We proceed with the details. Again,
$$B_m(tw)
=B_m(c\tilde{t}\inv \tilde{t} ctw c\tilde{t}\inv \tilde{t} c)
=B_m^c( \tilde{t} ctw c\tilde{t}\inv)
=B_m^c( t' cwc),
$$
where $t'=\tilde{t}ctc (cwc\tilde{t}\inv (cwc)\inv).$ For $w\in \mathrm{B}_{l-1}(\SO_{2l}),$ we have $cwc=w$ and hence $t'=\tilde{t}ctcw\tilde{t}\inv w\inv$ and 
$
B_m^c( t' cwc)=B_m^c( t' w).$

If $l$ is even, then $w\tilde{t}\inv w\inv=\tilde{t}$ and hence $$t'=\mathrm{diag}(t_1,\dots,t_{l-1},\frac{t_l\inv}{4}, 4t_l,t_{l-1}\inv,\dots,t_1\inv).$$
The map $t_l\mapsto \frac{t_l\inv}{4}$ sends $t^\prime$ to $t$ and $\frac{1}{4}+\frac{1}{4}t_l+\frac{1}{16}t_l\inv$ to itself. The fixed points of the map $t_l\mapsto \frac{t_l\inv}{4}$ are $t_l=\pm\frac{1}{2}.$ Furthermore, the mapping is an involution. Therefore, we can partition $F^\times\backslash\{\pm\frac{1}{2}\}$ into two disjoint sets $A$ and $B$ such that if $t_l\in A$ then $\frac{t_l\inv}{4}\in B.$ Hence, we partition $T_l$ into two sets $A_l$ and $B_l$ such that $t\in A_l$ if and only if $t_l\in A.$ Note that in this case $tw=l_l(a)n_1 w_l n_2$ for $a\in \GL_l$ and $n_1,n_2\in V_l$ with
$$
a^*=\left(\begin{matrix}
0& \frac{1}{4}+\frac{1}{4}t_l+\frac{1}{16}t_l\inv \\
\mathrm{diag}(t_{l-1}\inv,\dots,t_{1}\inv)(w'')^* &0
\end{matrix}\right).
$$
The map $t_l\mapsto \frac{t_l\inv}{4}$ takes $a^*\mapsto a^*.$ Therefore,
\begin{align*}
&\,\sum_{w\in \mathrm{B}_{l-1}(\SO_{2l})}\int_{ T_l\times U_{\SO_{2l}}}
B_m(twu)\tilde{\xi}_{1-s}^{i,v}(w_{l,l} j(twu), I_l)dtdu \\
=&\,\sum_{w\in \mathrm{B}_{l-1}(\SO_{2l})}\int_{ A_l\times U_{\SO_{2l}}}
B_m(twu)\tilde{\xi}_{1-s}^{i,v}(w_{l,l} j(twu), I_l)dtdu\\
&\,
+\sum_{w\in \mathrm{B}_{l-1}(\SO_{2l})}\int_{ B_l\times U_{\SO_{2l}}}
B_m(twu)\tilde{\xi}_{1-s}^{i,v}(w_{l,l} j(twu), I_l)dtdu\\
=&\,\sum_{w\in \mathrm{B}_{l-1}(\SO_{2l})}\int_{ A_l\times U_{\SO_{2l}}}
B_m(twu)\tilde{\xi}_{1-s}^{i,v}(w_{l,l} j(twu), I_l)dtdu\\
&\,+\sum_{w\in \mathrm{B}_{l-1}(\SO_{2l})}\int_{ B_l\times U_{\SO_{2l}}} B_m^c(t'cwcu)\tilde{\xi}_{1-s}^{i,v}(w_{l,l} j(twu), I_l)dtdu \\
=&\,\sum_{w\in \mathrm{B}_{l-1}(\SO_{2l})}\int_{ A_l\times U_{\SO_{2l}}}
(B_m+B_m^c)(twu)\tilde{\xi}_{1-s}^{i,v}(w_{l,l} j(twu), I_l)dtdu.
\end{align*}

The case for $l$ odd is similar. We have $w\tilde{t}\inv w\inv=\tilde{t}\inv$ and hence $t'=ctc.$ The map $t_l\mapsto t_l\inv$ sends $t^\prime$ to $t$ and $\frac{1}{2}(\frac{1}{2}-\frac{1}{4}(t_l+t_l\inv))$ to itself. The fixed points of the map $t_l\mapsto t_l\inv$ are $t_l=\pm 1.$ Furthermore, the mapping is an involution. Therefore, we can partition $F^\times\backslash\{\pm 1\}$ into two disjoint sets $A$ and $B$ such that if $t_l\in A$ then $t_l\inv\in B.$ Furthermore we partition $T_l$ into two sets $A_l$ and $B_l$ such that $t\in A_l$ if and only if $t_l\in A.$ Note that in this case $tw=l_l(a)n_1 w_l n_2$ for $a\in \GL_l$ and $n_1,n_2\in V_l$ with
$$
a^*=\left(\begin{matrix}
0& \frac{1}{2}(\frac{1}{2}-\frac{1}{4}(t_l+t_l\inv)) \\
\mathrm{diag}(t_{l-1}\inv,\dots,t_{1}\inv)(w'')^* &0
\end{matrix}\right).
$$
The map $t_l\mapsto t_l\inv$ takes $a^*\mapsto a^*.$ Therefore,
\begin{align*}
&\,\sum_{w\in \mathrm{B}_{l-1}(\SO_{2l})}\int_{ T_l\times U_{\SO_{2l}}}
B_m(twu)\tilde{\xi}_{1-s}^{i,v}(w_{l,l} j(twu), I_l)dtdu \\
=&\,\sum_{w\in \mathrm{B}_{l-1}(\SO_{2l})}\int_{ A_l\times U_{\SO_{2l}}}
(B_m+B_m^c)(twu)\tilde{\xi}_{1-s}^{i,v}(w_{l,l} j(twu), I_l)dtdu.
\end{align*}
Thus, we have 
\begin{align*}
0&=\sum_{w\in \mathrm{B}_l(\SO_{2l})}\int_{ T_{\SO_{2l}}\times U_{\SO_{2l}}}
(B_m+B_m^c)(twu)\tilde{\xi}_{1-s}^{i,v}(w_{l,l} j(twu), I_l)dtdu\\
&\quad
+\sum_{w\in \mathrm{B}_{l-1}(\SO_{2l})}\int_{ A_l\times U_{\SO_{2l}}}
(B_m+B_m^c)(twu)\tilde{\xi}_{1-s}^{i,v}(w_{l,l} j(twu), I_l)dtdu.
\end{align*}

Next, we define a function $f$ on a subset of $\GL_l$ so that we may apply Proposition \ref{JS Prop}. The first step in this is to describe the $a$'s which appear in the $W_v^*$'s. We recall the setup. For $w\in\mathrm{B}_l^c(\SO_{2l})$, by \cite[Proposition 4.5]{HL22}, there exists $w^\prime\in W(\GL_l)$ such that $w=t_l(w^\prime)\tilde{w}_l$ where $w'=(w_{i,j}')_{i,j=1}^l\in W(\GL_l)$ with $w_{l,1}'=0.$ We also let $(w')^*=(w_{i,j}'^*)_{i,j=1}^l$ and $r$ be such that $w_{r,1}'=1.$ 

First, suppose that $l$ is even. Let 
$$
a_2^*=\left(\begin{matrix}
\frac{t_l\inv(w_{1,j}'^*)_{j=1}^{l-1}}{4} & \frac{1}{4} \\
\mathrm{diag}(t_{l-1}\inv,\dots, t_1\inv) (w_{i,j}'^*)_{i=2,j=1}^{i=l,j=l-1} & \mathrm{diag}(t_{l-1}\inv,\dots, t_1\inv)\frac{(w_{i,l}'^*)_{i=2}^l}{4} \\
\end{matrix}\right).
$$
Then,
$$
a_2^*=\left(\begin{matrix}
1 & & t_r & & \\
  & \ddots & & & \\
  & & 1 & & \\
  & & & \ddots & \\
  & & & & 1
\end{matrix}\right)
 \mathrm{diag}\left(\frac{t_l\inv}{4},t_{l-1}\inv,\dots, t_{r+1}\inv, \frac{t_r\inv}{4},t_{r-1}\inv,\dots,t_1\inv\right)
w'^*.
$$
where $t_r$ is the $(1,l-r+1)$ entry of the unipotent matrix. Thus, we obtain
$$
a_2=\left(\begin{matrix}
1 & &  & & \\
  & \ddots & & & \\
  & & 1 & & -t_r \\
  & & & \ddots & \\
  & & & & 1
\end{matrix}\right) \mathrm{diag}\left(t_1,\cdots,t_{r-1}, 4t_r,t_{r+1},\cdots,t_{l-1},4t_l \right) w'.
$$
where $-t_r$ is the $(r,l)$ entry of the unipotent matrix. This determines $a$ on the $\mathrm{B}_{l}(\SO_{2l})$ sum. Let $t_{w'}\in T_{\GL_l}$ be such that $$t_{w'}\mathrm{diag}\left(t_1,\cdots,t_{r-1}, 4t_r,t_{r+1},\cdots,t_{l-1},4t_l \right)=\mathrm{diag}\left(t_1,\cdots,t_{r-1}, t_r,t_{r+1},\cdots,t_{l-1},t_l \right).$$ That is, $t_{w'}$ is a diagonal matrix consisting of $1$'s on the diagonal, except in the $(r.r)$ and $(l,l)$ coordinates where it is $\frac{1}{4}$.

Next, we consider the $a$ in the $\mathrm{B}_{l-1}(\SO_{2l})$ sum. Let $w^\prime=\left(\begin{matrix}
0& w'' \\
1 & 0
\end{matrix}\right)$ and
$$
a_3^*=\left(\begin{matrix}
0& \frac{1}{4}+\frac{1}{4}t_l+\frac{1}{16}t_l\inv \\
\mathrm{diag}(t_{l-1}\inv,\dots,t_{1}\inv)(w'')^* &0
\end{matrix}\right).
$$
Then,
$$
a_3^*=\left(\begin{matrix}
\frac{1}{4}+\frac{1}{4}t_l+\frac{1}{16}t_l\inv & & & \\
& t_{l-1}\inv & & \\
& & \ddots & \\
& & & t_{1}\inv 
\end{matrix}\right)
(w')^*.
$$
Hence
$$
a_3=\left(\begin{matrix}
t_1 & & & \\
& \ddots & & \\
& & t_{l-1} & \\
& & & (\frac{1}{4}+\frac{1}{4}t_l+\frac{1}{16}t_l\inv)\inv
\end{matrix}\right)
w'.
$$
Recall we partitioned $F^\times\backslash\{\pm\frac{1}{2}\}$ into two disjoint sets $A$ and $B$ such that if $t_l\in A$ then $\frac{t_l\inv}{4}\in B.$ Suppose $t_l, s_l\in A$ with $\frac{1}{4}+\frac{1}{4}t_l+\frac{1}{16}t_l\inv=\frac{1}{4}+\frac{1}{4}s_l+\frac{1}{16}s_l\inv$. This gives a quadratic equation in $s_l$ whose roots are $s_l=t_l$ and $s_l=\frac{t_l\inv}{4}.$ Since $s,t \in A$ it follows that we must have $t_l=s_l.$ Let 
$$A_{GL_l}=\left\{\left(\begin{matrix}
t_1 & & & \\
& \ddots & & \\
& & t_{l-1} & \\
& & & (\frac{1}{4}+\frac{1}{4}t_l+\frac{1}{16}t_l\inv)\inv
\end{matrix}\right) \, | \, t_1,\dots,t_{l-1}\in F^\times, t_l\in A \right\}.$$
The following map is well defined on $A_{GL_l}$: 

$$
\xi\left(\begin{matrix}
t_1 & & & \\
& \ddots & & \\
& & t_{l-1} & \\
& & & (\frac{1}{4}+\frac{1}{4}t_l+\frac{1}{16}t_l\inv)\inv
\end{matrix}\right):=\mathrm{diag}(t_1,\dots,t_l).
$$

Let $u=(u_{i,j})_{i,j=1}^l.$ Note that $u_{l,l+1}=0.$ Then, the embedding of $u$ in $\SO_{2l+1}$ is

$$
j(u)=\left(\begin{matrix}
(u_{i,j})_{i,j=1}^{i,j=l-1} & \left(\begin{matrix}
\frac{(u_{i,l})_{i=1}^{l-1}}{4}-\frac{(u_{i,l+1})_{i=1}^{l-1}}{2} & * & * \end{matrix}\right) & * \\
 & I_3 & * \\
 & & *
\end{matrix}\right).
$$
Let $\tilde{u}=\left(\begin{matrix}
(u_{i,j})_{i,j=1}^{i,j=l-1} &
\frac{(u_{i,l})_{i=1}^{l-1}}{4}-\frac{(u_{i,l+1})_{i=1}^{l-1}}{2} \\
0 & 1
\end{matrix}\right).$ Then $j(u)=l_l(\tilde{u})n_3.$ where $n_3\in V_l.$ The embedding takes $twu$ to $ l_l(a) n_1 w_l n_2 l_l(\tilde{u})n_3= n_4 l_l(a \tilde{u}^*) w_l n_5$ where $n_4, n_5\in V_l.$ Thus, by Equation \eqref{5.4}, $$\tilde{\xi}_{1-s}^{i,v}(w_{l,l} j(twu), I_l)=\mathrm{Vol}(\ol{V}_{l,i})|\mathrm{det}(a)|^{\frac{1}{2}-s}W_v^*(\mathrm{diag}(\frac{1}{2},\dots,\frac{1}{2})a \tilde{u}^*).$$

Next, we define a function on a subset of $\GL_l$ using its Bruhat decomposition. Specifically, we partition the Weyl group of $W(\GL_l)$ into $W_1(\GL_l)$ and $W_2(\GL_l)$ where $w'\in W_1(\GL_l)$ if $w^\prime\neq\left(\begin{matrix}
& w'' \\
1 & 
\end{matrix}\right)$ for any $w''\in W(\GL_{l-1}).$  and $w'\in W_2(\GL_l)$ if $w^\prime=\left(\begin{matrix}
& w'' \\
1 & 
\end{matrix}\right)$ for some $w''\in W(\GL_{l-1}).$ By \cite[Proposition 4.5]{HL22}, we have $w=t_l(w')\tilde{w}_l\in \mathrm{B}_{l}(\SO_{2l})$ if $w'\in W_1(\GL_l)$ and $w=t_l(w')\tilde{w}_l\in \mathrm{B}_{l-1}(\SO_{2l})$ if $w'\in W_2(\GL_l).$ Let
$$
X_l=\left(\bigsqcup_{w'\in W_1(\GL_l)} U_{\GL_l}T_{\GL_l} w' U_{\GL_l}\right)
\bigsqcup
\left(\bigsqcup_{w'\in W_2(\GL_l)} U_{\GL_l}A_{\GL_l} w' U_{\GL_l}\right).
$$

Recall the definition of $t_{w'}$ above. For $g=u_1 t w' u_2\in U_{\GL_l}T_{\GL_l} w' U_{\GL_l}$ with
$w'\in W_1(\GL_l)$ we define 
$$f(g)=
(B_m+B_m^c)(m_l(\mathrm{diag}(2,\dots,2) m_l(u_1 t_{w'} t w' u_2)\tilde{w}_l).
$$ 
For $g=u_1 t w' u_2 \in U_{\GL_l}A_{\GL_l} w' U_{\GL_l}$ with
$w'\in W_2(\GL_l)$ we define $$f(g)=
(B_m+B_m^c)(m_l(\mathrm{diag}(2,\dots,2))m_l(u_1 \xi(t) w' u_2)\tilde{w}_l).
$$
For $g\notin X_l$, set $f(g)=0.$
Then, $f(ug)=\psi(u)f(g)$ for any $u\in U_{\GL_l}$ and $g\in \GL_l.$ Also, $m_l(g)\tilde{w}_l=\tilde{w}_l m_l(g^*)$ for any $g\in\GL_l$ (we are still assuming that $l$ is even). Therefore, 
\begin{align*}
0&=\sum_{w\in W_1(\GL_l)}\int_{ T_{\GL_{l}}\times U_{\SO_{2l}}}
f(\mathrm{diag}(\frac{1}{2},\dots,\frac{1}{2})tw\tilde{u}^*) W_v^*(\mathrm{diag}(\frac{1}{2},\ldots,\frac{1}{2})tw\tilde{u}^*)\\
&\quad \quad \cdot |\mathrm{det}(t)|^{\frac{1}{2}-s-l}dtdu\\
&\quad
+\sum_{w\in W_2(\GL_l)}\int_{ A_{\GL_l}\times U_{\SO_{2l}}}
f(\mathrm{diag}(\frac{1}{2},\dots,\frac{1}{2})tw\tilde{u}^*) W_v^*(\mathrm{diag}(\frac{1}{2},\ldots,\frac{1}{2})tw\tilde{u}^*)\\
&\quad\quad \cdot |\mathrm{det}(t)|^{\frac{1}{2}-s-l}dtdu\\
&=\int_{U_{\GL_l}\backslash\GL_l} f(g)W_v^*(g)|\mathrm{det}(g)|^{\frac{1}{2}-s-l} |2|^{l(\frac{1}{2}-s-l)}dg.
\end{align*}
Thus, by Proposition \ref{JS Prop}, $f$ must identically vanish on $\GL_l$ and hence $X_l$. Therefore, for $l$ even, we have $(B_m+B_m^c)(tw)=0$  for any $t\in\SO_{2l}$ and $w\in \mathrm{B}_l(\SO_{2l}).$ Also, $(B_m+B_m^c)(tw)=0$ for any $w\in\mathrm{B}_{l-1}(\SO_{2l})$ and $t\in A_{l}$. Conjugation by $\tilde{t}c$ gives $(B_m+B_m^c)(tw)=0$ for any $w\in\mathrm{B}_{l-1}(\SO_{2l})$ and $t\in B_{l}.$

Suppose that $l$ is odd. Let 
$$
a_2^*=\left(\begin{matrix}
\frac{t_l\inv(w_{1,j}'^*)_{j=1}^{l-1}}{4} & \frac{1}{4} \\
\mathrm{diag}(t_{l-1}\inv,\dots, t_1\inv) (w_{i,j}'^*)_{i=2,j=1}^{i=l,j=l-1} & \mathrm{diag}(t_{l-1}\inv,\dots, t_1\inv)\frac{-(w_{i,l}'^*)_{i=2}^l}{2} \\
\end{matrix}\right).
$$
Then,
$$
a_2^*=\left(\begin{matrix}
1 & & \frac{-t_r}{2} & & \\
  & \ddots & & & \\
  & & 1 & & \\
  & & & \ddots & \\
  & & & & 1
\end{matrix}\right)
 \mathrm{diag}\left(\frac{t_l\inv}{4},t_{l-1}\inv,\dots, t_{r+1}\inv, \frac{-t_r\inv}{2},t_{r-1}\inv,\dots,t_1\inv\right)
w'^*.
$$
where $t_r$ is the $(1,l-r+1)$ entry of the unipotent matrix. Thus, we obtain
$$
a_2=\left(\begin{matrix}
1 & &  & & \\
  & \ddots & & & \\
  & & 1 & & \frac{t_r}{2} \\
  & & & \ddots & \\
  & & & & 1
\end{matrix}\right) \mathrm{diag}\left(t_1,\cdots,t_{r-1}, -2t_r.t_{r+1}.\cdots.t_{l-1},4t_l \right) w'.
$$
where $-t_r$ is the $(r,l)$ entry of the unipotent matrix. This determines $a$ on the $\mathrm{B}_{l}(\SO_{2l})$ sum. Let $t_{w'}\in T_{\GL_l}$ be such that $$t_{w'}\mathrm{diag}\left(t_1,\cdots,t_{r-1}, -2t_r,t_{r+1},\cdots,t_{l-1},4t_l \right)=\mathrm{diag}\left(t_1,\cdots,t_{r-1}, t_r,t_{r+1},\cdots,t_{l-1},t_l \right).$$ That is, $t_{w'}$ is a diagonal matrix consisting of $1$'s on the diagonal, except in the $(r.r)$ and $(l,l)$ coordinates where it is $\frac{-1}{2}$ and $\frac{1}{4}$ respectively.

Next, we consider the $a$ in the $\mathrm{B}_{l-1}(\SO_{2l})$ sum. Let $w^\prime=\left(\begin{matrix}
& w'' \\
1 & 
\end{matrix}\right)$ and
$$
a_3^*=\left(\begin{matrix}
& \frac{1}{2}(\frac{1}{2}-\frac{1}{4}(t_l+t_l\inv)) \\
\mathrm{diag}(t_{l-1}\inv,\dots,t_{1}\inv)(w'')^* &
\end{matrix}\right).
$$
Then,
$$
a_3^*=\left(\begin{matrix}
\frac{1}{2}(\frac{1}{2}-\frac{1}{4}(t_l+t_l\inv)) & & & \\
& t_{l-1}\inv & & \\
& & \ddots & \\
& & & t_{1}\inv 
\end{matrix}\right)
(w')^*.
$$
Hence
$$
a_3=\left(\begin{matrix}
t_1 & & & \\
& \ddots & & \\
& & t_{l-1} & \\
& & & (\frac{1}{2}(\frac{1}{2}-\frac{1}{4}(t_l+t_l\inv)))\inv
\end{matrix}\right)
w'.
$$
Recall we partitioned $F^\times\backslash\{\pm\frac{1}{2}\}$ into two disjoint sets $A$ and $B$ such that if $t_l\in A$ then $t_l\inv\in B.$ Suppose $t_l, s_l\in A$ with $\frac{1}{2}(\frac{1}{2}-\frac{1}{4}(t_l+t_l\inv))=\frac{1}{2}(\frac{1}{2}-\frac{1}{4}(s_l+s_l\inv))$. This gives a quadratic equation in $s_l$ whose roots are $s_l=t_l$ and $s_l=t_l\inv.$ Since $s,t \in A$ it follows that we must have $t_l=s_l.$ Let 
$$A_{GL_l}=\left\{\left(\begin{matrix}
t_1 & & & \\
& \ddots & & \\
& & t_{l-1} & \\
& & & (\frac{1}{2}(\frac{1}{2}-\frac{1}{4}(t_l+t_l\inv)))\inv
\end{matrix}\right) \, | \, t_1,\dots,t_{l-1}\in F^\times, t_l\in A \right\}.$$
The following map is well defined on $A_{GL_l}$: 

$$
\xi\left(\begin{matrix}
t_1 & & & \\
& \ddots & & \\
& & t_{l-1} & \\
& & & (\frac{1}{2}(\frac{1}{2}-\frac{1}{4}(t_l+t_l\inv)))\inv
\end{matrix}\right):=\mathrm{diag}(t_1,\dots,t_l).
$$

Let $u=(u_{i,j})_{i,j=1}^l.$ Note that $u_{l,l+1}=0.$ Then, the embedding of $u$ in $\SO_{2l+1}$ is

$$
j(u)=\left(\begin{matrix}
(u_{i,j})_{i,j=1}^{i,j=l-1} & \left(\begin{matrix}
\frac{(u_{i,l})_{i=1}^{l-1}}{4}-\frac{(u_{i,l+1})_{i=1}^{l-1}}{2} & * & * \end{matrix}\right) & * \\
 & I_3 & * \\
 & & *
\end{matrix}\right).
$$
Let $\tilde{u}=\left(\begin{matrix}
(u_{i,j})_{i,j=1}^{i,j=l-1} &
\frac{(u_{i,l})_{i=1}^{l-1}}{4}-\frac{(u_{i,l+1})_{i=1}^{l-1}}{2} \\
0 & 1
\end{matrix}\right).$ Then $j(u)=l_l(\tilde{u})n_3.$ where $n_3\in V_l.$ The embedding takes $twu$ to $ l_l(a) n_1 w_l n_2 l_l(\tilde{u})n_3= n_4 l_l(a \tilde{u}^*) w_l n_5$ where $n_4, n_5\in V_l.$ Thus, by Equation \eqref{xi intertwined}, $$\tilde{\xi}_{1-s}^{i,v}(w_{l,l} j(twu), I_l)=\mathrm{Vol}(\ol{V}_{l,i})|\mathrm{det}(a)|^{\frac{1}{2}-s}W_v^*(\mathrm{diag}(\frac{1}{2},\dots,\frac{1}{2})a \tilde{u}^*).$$

Next, we define a function on of $\GL_l$ using its Bruhat decomposition. Specifically, we partition the Weyl group of $W(\GL_l)$ into two sets, $W_1(\GL_l)$ and $W_2(\GL_l)$, where $w'\in W_1(\GL_l)$ if $w^\prime\neq\left(\begin{matrix}
& w'' \\
1 & 
\end{matrix}\right)$ for any $w''\in W(\GL_{l-1})$ and $w'\in W_2(\GL_l)$ if $w^\prime=\left(\begin{matrix}
& w'' \\
1 & 
\end{matrix}\right)$ for some $w''\in W(\GL_{l-1}).$ By \cite[Proposition 4.5]{HL22}, we have $w=t_l(w')\tilde{w}_l\in \mathrm{B}_{l}(\SO_{2l})$ if $w'\in W_1(\GL_l)$ and $w=t_l(w')\tilde{w}_l\in \mathrm{B}_{l-1}(\SO_{2l})$ if $w'\in W_2(\GL_l).$ Let
$$
X_l=\left(\bigsqcup_{w'\in W_1(\GL_l)} U_{\GL_l}T_{\GL_l} w' U_{\GL_l}\right)
\bigsqcup
\left(\bigsqcup_{w'\in W_2(\GL_l)} U_{\GL_l}A_{\GL_l} w' U_{\GL_l}\right).
$$

Recall the definition of $t_{w'}$ above. For $g=u_1 t w' u_2\in U_{\GL_l}T_{\GL_l} w' U_{\GL_l}$ with
$w'\in W_1(\GL_l)$ we define 
$$f(g)=
(B_m+B_m^c)(m_l(\mathrm{diag}(2,\dots,2) m_l(u_1 t_{w'} t w' u_2)\tilde{w}_l).
$$ 
For $g=u_1 t w' u_2 \in U_{\GL_l}A_{\GL_l} w' U_{\GL_l}$ with
$w'\in W_2(\GL_l)$ we define $$f(g)=
(B_m+B_m^c)(m_l(\mathrm{diag}(2,\dots,2))m_l(u_1 \xi(t) w' u_2)\tilde{w}_l).
$$
For $g\notin X_l$, we set $f(g)=0.$
Then, $f(ug)=\psi(u)f(g)$ for any $u\in U_{\GL_l}$ and $g\in \GL_l.$ Let $g=(g_{i,j})_{i,j=1}^l$ and $g^*=(g^*_{i,j})_{i,j=1}^l$. Then $$
m_l(g)\tilde{w}_l=\tilde{w}_l
\left(\begin{matrix}
(g_{i,j})_{i,j=1}^{l-1} & 0 & (g_{i,l})_{i=1}^{l-1} & 0 \\
0 & g_{1,1}^* & 0 & (g_{1,j}^*)_{j=2}^{l} \\
(g_{l,j})_{j=1}^{l-1} & 0 & g_{l,l} & 0 \\
0 & (g_{i,1}^*)_{i=2}^{l} & 0 & (g_{i,j}^*)_{i,j=2}^{l} \\
\end{matrix}
\right).
$$
In particular, if $u\in U_{\GL_l}$, then 
$$
m_l(u)\tilde{w}_l=\tilde{w}_l
\left(\begin{matrix}
(g_{i,j})_{i,j=1}^{l-1} & 0 & (g_{i,l})_{i=1}^{l-1} & 0 \\
 & 1 & 0 & (g_{1,j}^*)_{j=2}^{l} \\
 &  & 1 & 0 \\
 & &  & (g_{i,j}^*)_{i,j=2}^{l} \\
\end{matrix}
\right).
$$
and the last matrix is upper triangular.
Therefore, 
\begin{align*}
0&=\sum_{w\in W_1(\GL_l)}\int_{ T_{\GL_{l}}\times U_{\SO_{2l}}}
f(\mathrm{diag}(\frac{1}{2},\dots,\frac{1}{2})tw\tilde{u}^*) W_v^*(\mathrm{diag}(\frac{1}{2},\ldots,\frac{1}{2})tw\tilde{u}^*)\\
&\quad \quad \cdot |\mathrm{det}(t)|^{\frac{1}{2}-s-l}dtdu\\
&\quad
+\sum_{w\in W_2(\GL_l)}\int_{ A_{\GL_l}\times U_{\SO_{2l}}}
f(\mathrm{diag}(\frac{1}{2},\dots,\frac{1}{2})tw\tilde{u}^*) W_v^*(\mathrm{diag}(\frac{1}{2},\ldots,\frac{1}{2})tw\tilde{u}^*)\\
&\quad \quad \cdot |\mathrm{det}(t)|^{\frac{1}{2}-s-l}dtdu\\
&=\int_{U_{\GL_l}\backslash\GL_l} f(g)W_v^*(g)|\mathrm{det}(g)|^{\frac{1}{2}-s-l}|2|^{l(\frac{1}{2}-s-l)}dg.
\end{align*}
Thus, by Proposition \ref{JS Prop}, $f$ must identically vanish on $\GL_l$ and hence $X_l$. Therefore, for $l$ odd, we have $(B_m+B_m^c)(tw)=0$  for any $t\in\SO_{2l}$ and $w\in \mathrm{B}_l(\SO_{2l}).$ Also, $(B_m+B_m^c)(tw)=0$ for any $w\in\mathrm{B}_{l-1}(\SO_{2l})$ and $t\in A_{l}$. Conjugation by $\tilde{t}c$ gives $(B_m+B_m^c)(tw)=0$ for any $w\in\mathrm{B}_{l-1}(\SO_{2l})$ and $t\in B_{l}$ for $l$ odd.

Finally, for any $l$, we have 
$(B_m+B_m^c)(tw)=0$  for any $t\in\SO_{2l}$ and $w\in \mathrm{B}_l(\SO_{2l}).$  Also, 
\begin{align*}
    B_m(tw)
&=B_m(c\tilde{t}\inv \tilde{t} ctw c\tilde{t}\inv \tilde{t} c)
=B_m^c( \tilde{t} ctw c\tilde{t}\inv)=B_m^c( t' cwc),
\end{align*}
where $t'=\tilde{t}ctc (cwc\tilde{t}\inv (cwc)\inv).$ It follows that $(B_m+B_m^c)(tw)=0$  for any $t\in\SO_{2l}$ and $w\in \mathrm{B}_l^c(\SO_{2l}).$ From Proposition \ref{l-2 prop}, Theorem \ref{l-1 thm}, and Corollaries \ref{conj gamma1} and \ref{conj gamma2} and Lemma \ref{6.3}(3), it follows that we have, for any $g\in\SO_{2l},$ $(B_m+B_m^c)(g)=0.$
This completes the proof of Theorem \ref{l thm}.
\end{proof}

Similarly, we have a corollary on the equality of twisted gamma factors of $\pi$ and $c \cdot \pi$ as follows, which again can be implied by the work of \cite{Art13, JL14}. 

\begin{cor}\label{conj gamma3}
Let $\pi$ be an irreducible  $\psi$-generic supercuspidal representation of $\SO_{2l}$. Then $\gamma(s, \pi\times\tau,\psi)=\gamma(s, c\cdot\pi\times\tau,\psi)$ for all irreducible $\psi^{-1}$-generic representations $\tau$ of $\GL_{n}$ for any $n\leq l.$ 
\end{cor}

\begin{proof}
By Corollaries \ref{conj gamma1} and \ref{conj gamma2}, $\gamma(s, \pi\times\tau,\psi)=\gamma(s, c\cdot\pi\times\tau,\psi)$ for all irreducible $\psi^{-1}$-generic representations $\tau$ of $\GL_{k}$ for $k\leq l-1.$ Repeating the steps and the change of variables performed in the previous proof  show that the zeta integrals for $\pi$ and $c\cdot\pi$ are equal and hence so are the $\gamma$-factors for $\GL_l.$
\end{proof}

Finally, we can complete the proof of the converse theorem. From Theorem \ref{l thm}, for any $g\in\SO_{2l},$ we have
$$
B_m(g,f)-B_m(g,f')+B_m^c(g,f)-B_m^c(g,f')=0.$$ 
By the uniqueness of Whittaker models, we must have $\pi\cong\pi'$ or $\pi\cong c\cdot\pi'$. This completes the proof of Theorem \ref{converse thm intro}.

\end{document}